\documentclass[numbers,webpdf,imanum]{ima-authoring-template}%

\graphicspath{{Fig/}}

\usepackage{graphicx}
\usepackage{mathptmx}      
\usepackage{amsmath}
\usepackage{amssymb}
\usepackage{nicefrac}
\usepackage{caption}
\usepackage{subcaption}

\newtheorem{assumption}{Assumption}%
\newtheorem{corollary}{Corollary}%

\usepackage{xcolor}
\usepackage{hyperref}
\hypersetup{colorlinks=true, linkcolor=blue, citecolor=blue}
\makeatletter
\def\cl@chapter{\@elt {chapter}}
\makeatother
\usepackage[noabbrev,capitalize,nameinlink]{cleveref}

\crefformat{equation}{(#2#1#3)}
\crefrangeformat{equation}{(#3#1#4) to~(#5#2#6)}
\crefmultiformat{equation}{(#2#1#3)}{ and~(#2#1#3)}{, (#2#1#3)}{ and~(#2#1#3)}
\crefname{assumption}{Assumption}{Assumptions}
\newcommand{\N}{\mathbb{N}}
\newcommand{\R}{\mathbb{R}}
\newcommand{\C}{\mathbb{C}}
\newcommand{\F}{\mathcal{F}}
\newcommand{\Hc}{\mathcal{H}}

\newcommand{\sgn}{\mathrm{sgn}}
\newcommand\restr[2]{\ensuremath{\left.#1\right\vert_{#2}}}
\newcommand\abs[1]{\ensuremath{\left.\vert#1\right\vert}}
\newcommand\tensor{\,\tilde{\otimes}\,}

\theoremstyle{thmstyletwo}%
\newtheorem{theorem}{Theorem}
\newtheorem{example}{Example}%
\newtheorem{remark}{Remark}%
\newtheorem{lemma}{Lemma}%
\numberwithin{equation}{section}%
\newtheorem{method}{MDN method}%

\begin{document}

\DOI{}
\copyrightyear{}
\vol{}
\pubyear{}
\access{ }
\appnotes{Paper}
\copyrightstatement{}
\firstpage{1}

\title[Domain decomposition for quasilinear parabolic equations]{Linearly convergent nonoverlapping domain decomposition methods for quasilinear parabolic equations}
\author{Emil Engstr\"{o}m* and Eskil Hansen
\address{\orgdiv{Centre for Mathematical Sciences}, \orgname{Lund University}, \orgaddress{\street{P.O.\ Box 118}, \postcode{221 00}, \state{Lund}, \country{Sweden}}}}

\authormark{Emil Engstr\"{o}m and Eskil Hansen}

\corresp[*]{Corresponding author: \href{email:}{emil.engstrom@math.lth.se}}

\received{29}{8}{2023}

\abstract{
We prove linear convergence for a new family of modified Dirichlet--Neumann methods applied to quasilinear parabolic equations, as well as the convergence of the Robin--Robin method. Such nonoverlapping domain decomposition methods are commonly employed for the parallelization of partial differential equation solvers. Convergence has been extensively studied for elliptic equations, but in the case of parabolic equations there are hardly any convergence results that are not relying on strong regularity assumptions. Hence, we construct a new framework for analyzing domain decomposition methods applied to quasilinear parabolic problems, based on fractional time derivatives and time-dependent Steklov--Poincar\'e operators. The convergence analysis is conducted without assuming restrictive regularity assumptions on the solutions or the numerical iterates. We also prove that these continuous convergence results extend to the discrete case obtained when combining domain decompositions with space-time finite elements.
}
\keywords{nonoverlapping domain decompositions; quasilinear parabolic equations; linear convergence; time-dependent Steklov--Poincar\'e operators; space-time finite elements.}

\maketitle

\section{Introduction}\label{sec:intro}
Domain decomposition methods enable the usage of parallel and distributed hardware and are commonly employed when approximating the solutions to elliptic equations. The basic idea is to first decompose the equation's domain into subdomains. The numerical method then consists of iteratively solving the elliptic equation on each subdomain and thereafter communicating the results via the boundaries to the neighboring subdomains. An in-depth survey of the topic can be found in the monographs~\cite{quarteroni,widlund}. 

A recent development in the field is to apply this approach to parabolic equations. The decomposition into spatial subdomains is then replaced by a decomposition into space-time cylinders. In general, space-time decomposition schemes enable additional parallelization and less storage requirements when combined with a standard numerical method for parabolic problems. The methods have especially gained attention in the contexts of parallel time integrators; surveyed in~\cite{gander15}, space-time finite elements; surveyed in~\cite{steinbach19}, and parabolic problems with a spatial domain given by a union of domains with very different material properties~\cite{japhet20,japhet16}.

There have been several studies concerning the convergence and other theoretical aspects of space-time decomposition methods applied to linear parabolic equations, especially for Schwarz waveform-relaxation (SWR) type metods. Results for one-dimensional or rectangular spatial domains have, e.g., been derived in the papers~\cite{gander07,kwok16,kwok21,keller02,lemarie13a,lemarie13b}. For more general domains, convergence has been proven for SWR methods applied linear parabolic equations in~\cite{Halpern12,japhet13}, semilinear parabolic equations in~\cite{Halpern10}, and quasilinear parabolic equations in~\cite{gander23}. However, all these convergence results rely on additional regularity assumptions on the solution of the parabolic problem, or even the SWR approximation itself, which are not necessarily fulfilled for spatial (sub)domains that are only Lipschitz.

In the setting of domain decomposition methods applied to elliptic equations it is also standard that the convergence results for the continuous case directly extend to the discrete case obtained when combining the domain decomposition method with a space discretization, e.g., finite elements. This does not seem to hold true for the parabolic frameworks with more general spatial domains. The only related results stated in the above references  is the convergence of SWR decompositions combined with a semidiscretization in time via a discontinuous Galerkin scheme~\cite{Halpern12,japhet13}. 

Hence, the main goal of this study is to derive a new framework for nonoverlapping space-time domain decompositions for parabolic equations that enables the derivation of methods with the features below.
\begin{itemize}
\item The method is linearly convergent in the continuous case when applied to a family of quasilinear parabolic equations. 
\item The convergence analysis does not rely on additional regularity assumptions on the subdomains, the solution of parabolic problem, or the approximation itself.  
\item The continuous convergence result directly extends to the fully discrete setting obtained when combining the domain decomposition method with space-time finite elements.
\end{itemize}
We will furthermore strive to create a general enough framework such that the convergence of the Robin--Robin method also follows for quasilinear parabolic equations under mild regularity assumptions. That is, the same SWR method with zeroth-order transmission conditions employed in the quasilinear study~\cite{gander23}.

As a start, we introduce the notation
\begin{equation*}
Au=\partial_tu-\nabla\cdot\alpha(x, u, \nabla u)+\beta(x, u, \nabla u)
\end{equation*}
and consider the quasilinear parabolic equation
\begin{equation}\label{eq:strong+}
\left\{
     \begin{aligned}
            Au&=g  & &\text{in }\Omega\times(0,\infty),\\
            u&=\rho & &\text{on }\partial\Omega\times(0,\infty),\\
            u&=0 & &\text{in }\Omega\times\{0\},
        \end{aligned}
\right.
\end{equation}
where the spatial domain $\Omega\subset\R^d$, $d=2,3,\ldots$, is bounded with boundary $\partial\Omega$ and the functions $\alpha$ and $\beta$ are Lipschitz continuous and satisfy a uniform monotonicity property; see~\cref{sec:prob} for the precise assumptions. Note that the results of this paper also hold for $d=1$, but this case requires a slightly different setup. 

Next, we decompose the spatial domain $\Omega$ into nonoverlapping subdomains $\Omega_i$, $i=1,2$, with boundaries $\partial\Omega_i$, and denote the interface separating the subdomains $\Omega_i$ by $\Gamma$. That is, 
\begin{equation}\label{eq:domain}
\overline{\Omega}=\overline{\Omega}_1\cup\overline{\Omega}_2,\quad \Omega_1\cap\Omega_2=\emptyset,\quad\text{and}\quad\Gamma=(\partial\Omega_1\cap\partial\Omega_2)\setminus\partial\Omega. 
\end{equation}
The space-time cylinder $\Omega\times(0,\infty)$ is thereby decomposed into $\Omega_i\times(0,\infty)$, $i=1,2$, as illustrated in~\cref{fig:spacetimecylinder}. The current setting is also valid for spatial subdomains $\Omega_i$ given as unions of nonadjacent subdomains, i.e., 
\begin{displaymath}
\Omega_i=\cup_{\ell=1}^s\Omega_{i\ell}\quad\text{and}\quad
\overline{\Omega}_{i\ell}\cap\overline{\Omega}_{ij}=\emptyset\quad\text{for }\ell\neq j. 
\end{displaymath}
\begin{figure}
\centering
\includegraphics[width=.45\linewidth]{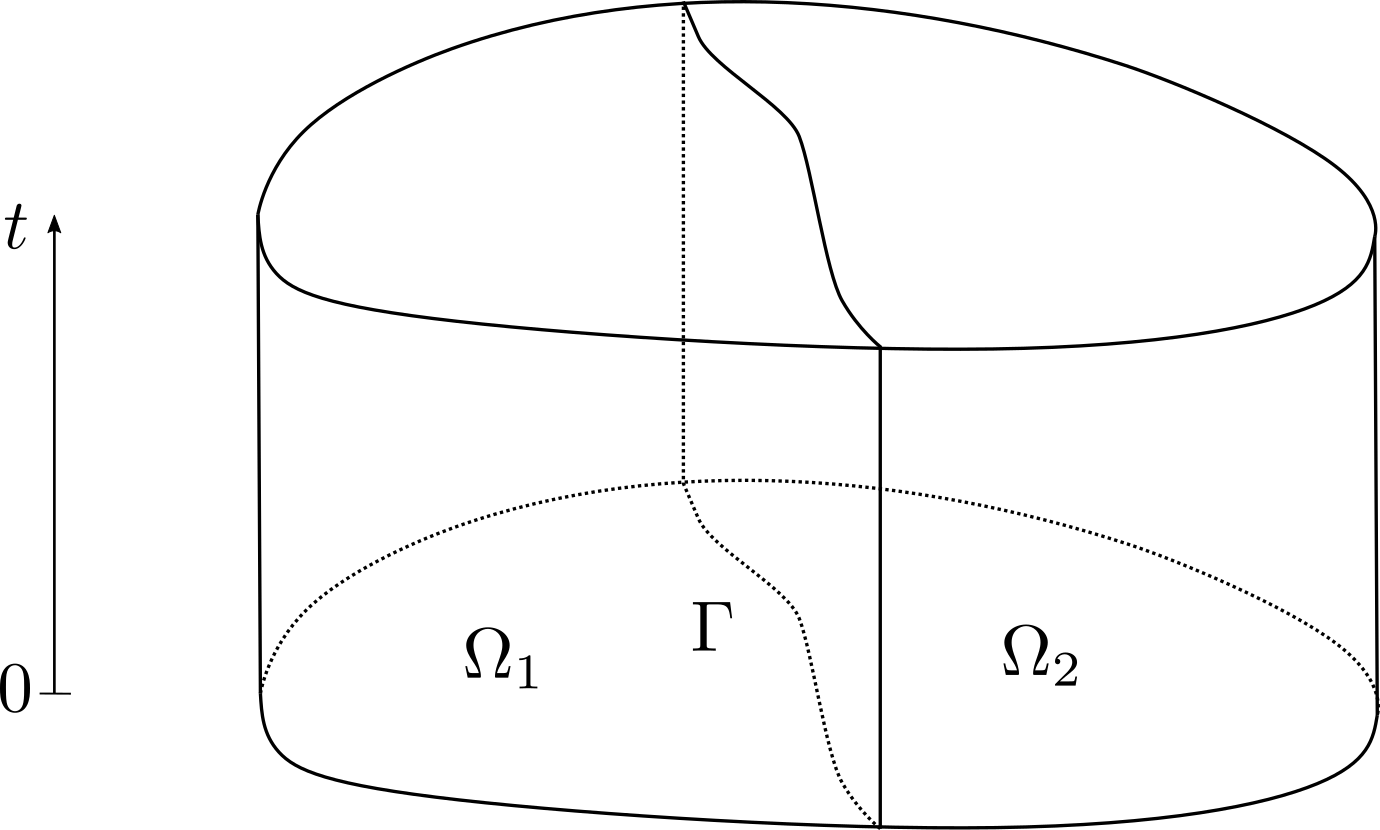}
\caption{The nonoverlapping decomposition of the space-time cylinder.}
\label{fig:spacetimecylinder}
\end{figure}%
With a fixed domain decomposition we can reformulate~\cref{eq:strong+} on $\Omega_i\times (0,\infty)$, $i=1,2$, connected via transmission conditions on $\Gamma\times(0,\infty)$. More precisely, we have the parabolic transmission problem
\begin{equation}\label{eq:TPStrong}
\left\{
     \begin{aligned}
            Au_i&=g  & &\text{in }\Omega_i\times(0,\infty), & &\\
             u_i&=\rho & &\text{on }(\partial\Omega_i\setminus\Gamma)\times(0,\infty), & &\\
             u_i&=0 & &\text{in }\Omega_i\times\{0\}, & &\text{ for }i=1,2,\\[5pt]
             u_1&=u_2  & &\text{on }\Gamma\times(0,\infty), & &\\
             \alpha(\nabla u_1)\cdot\nu_1 &= -\alpha(\nabla u_2)\cdot\nu_2  & &\text{on }\Gamma\times(0,\infty), & &
        \end{aligned}
\right.
\end{equation}
where $\nu_i$ denotes the unit outward normal vector of $\partial\Omega_i$. Alternating between the decomposed space-time cylinders and the transmission conditions generates the space-time generalizations of the standard domain decomposition methods for elliptic equations. 

Before introducing the space-time domain decomposition methods and their corresponding finite element discretizations, one needs to find a suitable functional analytic setting for the analysis. Observe that the standard variational framework for parabolic problems and the corresponding Petrov--Galerkin methods~\cite[Section~2.3]{steinbach19} are all based on trial spaces in the Bochner space intersection
\begin{displaymath}
H^1\bigl((0,\infty),H^{-1}(\Omega_i)\bigr)\cap L^2\bigl((0,\infty),H^{1}(\Omega_i)\bigr),
\end{displaymath}
which are, unfortunately, not well suited for our domain decompositions. The issue is that two functions $u_i$, $i=1,2$, in the above trial space, which coincide on $\Gamma\times\R$ in the sense of trace, can not be ``glued'' together into a new function in $H^1\bigl((0,\infty),H^{-1}(\Omega)\bigr)$; compare with~\cite[Example~2.14]{costabel90}. Hence, the transmission problem~\cref{eq:TPStrong} does not necessarily yield a solution to~\cref{eq:strong+} in this context. 

In order to remedy this, we consider the more general framework for parabolic problems with the trial/test spaces in 
\begin{displaymath}
H^{1/2}\bigl((0,\infty),L^2(\Omega_i)\bigr)\cap L^2\bigl((0,\infty),H^{1}(\Omega_i)\bigr),
\end{displaymath}
which originates from~\cite{lionsmagenes2} and resolves the above issue. In the context of space-time finite elements, this $H^{1/2}$-Bubnov--Galerkin setting was proposed by~\cite{fontesthesis}. The nonlocal fractional time derivatives arising in these numerical schemes can be effectively implemented, e.g., by introducing a temporal Galerkin basis with compact support in the frequency domain; see~\cite{Dahlgren04}, or by implementing an efficient evaluation of the temporal Hilbert transform~\cite{Langer21,steinbach20}. Note that the first approach requires the extension of the parabolic equation to all times $t\in\R$ and the second approach is given on finite time intervals. The $H^{1/2}$-framework has also been employed for space-time wavelet Galerkin discretizations~\cite{larsson15} and boundary element methods~\cite{costabel90}.

The rest of the paper is organized as follows: In~\cref{sec:prob} we state the precise problem formulation. We then derive the properties of the required function spaces and operators in~\cref{sec:prel,sec:tensor}. There are several technicalities associated with introducing the space-time trace operator acting on intersections of Bochner spaces, which are rarely considered in the numerical literature. We will therefore make an effort to state precise definitions and proofs. Next, we introduce the weak formulations in~\cref{sec:weak} and prove equivalence between the quasilinear parabolic equation and the transmission problem. The nonlinear time-dependent Steklov--Poincar\'e operators are analysed in~\cref{sec:SP}. These operators enable the interface reformulations of the transmission problem and the space-time domain decomposition methods. Based on the interface reformulations we introduce new modified Dirichlet--Neumann methods in~\cref{sec:lin}, and prove their linear convergence under minimal regularity assumptions via a variation of Zarantello's theorem. We also prove convergence of the Robin--Robin method in~\cref{sec:conv}, under mild regularity assumptions. Finally, the extension to the discrete space-time finite elements setting of~\cite{Dahlgren04,fontesthesis} and a set of numerical experiments are given in~\cref{sec:FEM}.

Throughout the paper $c$ and $C$ will denote generic positive constants.

\section{Problem setting}\label{sec:prob}
We make the following assumptions on the problem data $(\Omega,\Omega_i,\alpha,\beta,g,\rho)$ of~\cref{eq:strong+,eq:domain}.
\begin{assumption}\label{ass:domains}
The spatial domains $\Omega\subset \R^{d}$ and $\Omega_{i}$, $i=1,2$ are all bounded and Lipschitz. The spatial interface~$\Gamma$ and the sets $\partial\Omega\setminus\partial\Omega_{i}$, $i=1,2$, are all $(d-1)$-dimensional Lipschitz manifolds.
\end{assumption}
For a description of Lipschitz domains, see~\cite[Chapter 6.2]{kufner}. The assumptions are made in order ensure the existence of the spatial trace operator, as well as, to allow the usage of Poincar\'e's inequality.\\
\begin{assumption}\label{ass:eq}
The functions $\alpha$ and $\beta$ satisfy the conditions below, where $h_\ell$ denotes a given nonnegative function in~$L^\infty(\Omega)$. 
\begin{itemize}
\item $x\mapsto \alpha(x, y, z)$ and $x\mapsto \beta(x, y, z)$ are measurable on $\Omega$ for all $(y,z)\in\R\times\R^d$.
\item $\alpha$ and $\beta$ are Lipschitz continuous with respect to $(y,z)$. That is,
\begin{displaymath}
	|\alpha(x, y, z)-\alpha(x, y', z')|\leq h_1(x) \bigl(|z-z'|+|y-y'|\bigr)\quad\text{for all }y,y´\in\R,\, z,z'\in\R^d,
\end{displaymath}
and almost every $x\in\Omega$. The same holds for $\beta$.
\item $\alpha$ and $\beta$ satisfy the uniform monotonicity condition
\begin{displaymath}
	\bigl(\alpha(x, y, z)-\alpha(x, y', z')\bigr)\cdot(z-z')+\bigl(\beta(x, y, z)-\beta(x, y', z')\bigr)(y-y')\geq h_2(x)|z-z'|^2-h_3(x)|y-y'|^2 
\end{displaymath}
for all $y,y´\in\R$, $z,z'\in\R^d$, and almost every $x\in\Omega$. Here
\begin{equation}\label{eq:monbound}
	\inf_{x\in\Omega} h_2(x) > C_p \sup_{x\in\Omega} h_3(x),
\end{equation}
with $C_p$ denoting the largest Poincaré constant of $\Omega$ and $\Omega_i$, $i=1, 2$.
\end{itemize}
\end{assumption}
\begin{example}\label{ex:lin}
\emph{Consider a linear advection-diffusion-reaction equation in a heterogeneous media. That is, a parabolic equation governed by the vector field
\begin{displaymath}
Au=\partial_tu-\nabla\cdot\bigl(\alpha(x)\nabla u\bigr)+\beta(x)\cdot\nabla u + \gamma(x) u.
\end{displaymath}
\cref{ass:eq} then holds if the functions $\alpha,\beta,\gamma \in L^\infty(\Omega)$ fulfill~\cref{eq:monbound} with $h_2=\alpha-|\beta|^2/2$ and $h_3=|\gamma|^2+|\beta|^2/2$.}
\end{example}
\begin{example}
\emph{A quasilinear equation that satisfies~\cref{ass:eq} could have the form
\begin{displaymath}
        \alpha(x, u, \nabla u)=\nabla u+\gamma(x)\sin(|\nabla u|)(1,\ldots,1)^\mathrm{T}
\end{displaymath}
and $\beta(u)=\arctan(u)$, where $\gamma \in L^\infty(\Omega)$. \cref{ass:eq} is then valid if $h_2=1-\sqrt{d}\gamma$ and $h_3=1$ fulfill~\cref{eq:monbound}.}
\end{example}
\begin{assumption}\label{ass:f}
The source term $g$ is an element in $L^2\bigl(\Omega\times(0,\infty)\bigr)$ and the boundary value $\rho$ is an element in 
$H^{1/4}\bigl((0,\infty),L^2(\partial\Omega\bigr)\bigr)\cap L^2\bigl((0,\infty), H^{1/2}(\partial\Omega)\bigr)$.
\end{assumption}

With these assumptions we can apply the nonlinear $H^{1/2}$-framework of~\cite{fontes09} by extending the original parabolic problem~\cref{eq:strong+} into a more general class of equations given for all times $t\in\R$. We will only give a short summary of this procedure, and we refer to  to~\cite{fontes09,larsson15} for precise definitions and proofs. 

If~\cref{ass:domains,ass:eq,ass:f} hold, then the quasilinear parabolic equation~\cref{eq:strong+} has a unique weak solution $u^+\in H^{1/2}\bigl((0,\infty),L^2(\Omega)\bigr)\cap L^2\bigl((0,\infty),H^{1}(\Omega)\bigr)$, with its spatial trace equal to~$\rho$ and fulfilling the bound
\begin{displaymath}
    	\int_0^\infty\frac1t\|u^+(t)\|^2_{L^2(\Omega)}\,\mathrm{d}t<\infty.
\end{displaymath}
The latter is a weak interpretation of the homogeneous initial value, i.e., $u^+$ decays sufficiently rapidly to zero as~$t$ tends to $0^+$. 

Let $e$ denote the extension by zero of measurable functions on $\Omega\times (0,\infty)$ to $\Omega\times\R$, or on $\partial\Omega\times (0,\infty)$ to $\partial\Omega\times\R$. Due to the decay of $u^+$ at time zero and the fact that the temporal regularity of $g$ and $\rho$ are both stated in $H^s$, with $s<1/2$, the extended functions $(eu^+,eg,e\rho)$ all retain their spatial and temporal regularity. Furthermore, there exists, a non-unique, function $w\in H^{1/2}\bigl(\R,L^2(\Omega)\bigr)\cap L^2\bigl(\R,H^{1}(\Omega)\bigr)$ with its spatial trace equal to $e\rho$. Hence, the functions $(u,f)=(eu^+-w,eg)$ satisfy the extended parabolic equation 
\begin{equation}\label{eq:strong}
\left\{
     \begin{aligned}
            A(u+w)&=f  & &\text{in }\Omega\times\R,\\
            u&=0 & &\text{on }\partial\Omega\times\R.
        \end{aligned}
\right.
\end{equation}
Note that non-homogeneous initial values can also be included via a similar subtraction approach as for the space-time dependent boundary condition~$\rho$; see~\cite[Theorem~4.5]{fontes09}. However, we will refrain from adding this additional layer of technicalities in the proceeding analysis. 

The domain decomposition methods considered in the rest of this study will approximate solutions for the extended class of parabolic equations~\cref{eq:strong}, which enables the derivation of a rigorous convergence analysis. Note that the nonlinear differential operators $A$ and $A(\cdot+w)$ fulfill the very same properties, especially Lipschitz continuity and uniform monotonicity, required for the rest of the analysis; compare with the proof of~\cref{lemma:Fi}. Hence, for notational simplicity we will conduct the rest of the analysis for the case $w=\rho=0$ and with an arbitrary $f\in L^2(\Omega\times\R)$.

\begin{remark}
\emph{Observe that the extended problem~\cref{eq:strong} does \emph{not} involve a backward diffusion equation. Instead, the physical interpretation is that we consider every (forward) diffusion processes, including~\cref{eq:strong+}, that starts with zero concentration $u$ at time ``$t=-\infty$'', evolves over any given finite time interval according to the source terms $(w,f)$, and decays to zero as time tends to infinity.}
\end{remark}

\section{Preliminaries}\label{sec:prel}
We first recall some definitions from functional analysis. Let $X,Y$ be Hilbert spaces and denote the dual of $X$ by $X^*$. The corresponding dual paring in $X^*\times X$ is denoted by $\langle\cdot, \cdot\rangle$. A form $a:X\times X\rightarrow \R$ is referred to as Lipschitz continuous if
\begin{displaymath} 
        |a(u, w)-a(v, w)|\leq C\|u-v\|_X\|w\|_X\quad \textrm{for all } u, v, w\in X.
\end{displaymath}
Similarly, we say that an operator $G:X\rightarrow Y$ is Lipschitz continuous if
\begin{displaymath} 
        \|Gu-Gv\|_Y\leq C\|u-v\|_X\quad \textrm{for all } u, v\in X.
\end{displaymath}
Note that for a form $a:X\times X\rightarrow\R$ that is linear and bounded in the second argument we can define the operator 
\begin{displaymath} 
    G:X\rightarrow X^*:u\mapsto a(u,\cdot).
\end{displaymath}
It is clear that the operator $G$ is Lipschitz continuous if and only if the corresponding form $a$ is Lipschitz continuous.

A form $a:X\times X\rightarrow \R$ is said to be uniformly monotone in $Y$, where $X\hookrightarrow Y$, if
\begin{displaymath} 
        a(u, u-v)-a(v, u-v)\geq c\|u-v\|_Y^2\quad \textrm{for all } u, v\in X.
\end{displaymath}Similarly, an operator $G:X\rightarrow X^*$ is uniformly monotone in $Y$, where $X\hookrightarrow Y$, if
\begin{displaymath} 
        \langle Gu-Gv, u-v\rangle\geq c\|u-v\|_Y^2\quad \textrm{for all } u, v\in X.
\end{displaymath}
As for Lipschitz continuity, a form $a$ is uniformly monotone in $Y$ if and only if the corresponding operator $G$ is uniformly monotone in $Y$. If $X=Y$ we simply say that $a$ and $G$ are uniformly monotone. 

A linear operator $P:X\rightarrow X^*$ that is uniformly monotone in $X$ is said to be coercive. Also, it is called symmetric if
\begin{displaymath}
    \langle P\eta, \mu\rangle=\langle P\mu, \eta\rangle\quad\text{for all } \eta,\mu\in X.
\end{displaymath}
For a linear isomorphism $Q:X\rightarrow X$ we define the adjoint
\begin{displaymath}
    Q^*:X^*\rightarrow X^*:u \mapsto\langle u, Q\cdot \rangle
\end{displaymath}
and note that $Q^*G$ is uniformly monotone if and only if $a(\cdot, Q\cdot)$ is uniformly monotone.

As a notational convention, operators only depending on space or time are ``hatted'' and their extensions to space-time are denoted without hats, e.g., 
\begin{displaymath}
\hat{\Delta}:H_0^1(\Omega)\rightarrow H^{-1}(\Omega)\quad\text{and}\quad \Delta:L^{2}\bigl(\R,H_0^1(\Omega)\bigr)\rightarrow L^{2}\bigl(\R,H^{-1}(\Omega)\bigr).
\end{displaymath}

Consider the spatial function spaces
\begin{displaymath}
V=H_0^1(\Omega),\quad  V_i^0=H_0^1(\Omega_i),\quad\text{and}\quad
V_i=\{v\in H^1(\Omega_i): \restr{(\hat{T}_{\partial\Omega_i}v)}{\partial\Omega_i\setminus\Gamma}=0\}.
\end{displaymath}
Here, $\hat{T}_{\partial\Omega_i}: H^1(\Omega_i)\rightarrow H^{1/2}(\partial\Omega_i)$ denotes the trace operator, see~\cite[Theorem 6.8.13]{kufner}. The norm on $V_i$ and $V_i^0$ is given by 
\begin{displaymath}
	    \|v\|_{V_i}=\bigl(\|\nabla v\|^2_{L^2(\Omega_i)^{d}}+\|v\|^2_{L^2(\Omega_i)}\bigr)^{1/2}.
\end{displaymath}
By Poincaré's inequality and \cref{ass:domains} we have that $v\mapsto\|\nabla v\|_{L^2(\Omega_i)^{d}}$ is an equivalent norm on $V_i$ and $V_i^0$. 
The Hilbert spaces $V$, $V_i$, and $V_i^0$ are all separable. The space $H^{1/2}(\partial\Omega_i)$ is defined as
\begin{equation}\label{eq:slobodetskii}
	\begin{aligned}
	    &H^{1/2}(\partial\Omega_i)=\{v\in L^2(\partial\Omega_i):  \|v\|_{H^{1/2}(\partial\Omega_i)}<\infty\}\quad\text{with}\\
	    &\|v\|_{H^{1/2}(\partial\Omega_i)} =
	                 \Big(\int_{\partial\Omega_i}\int_{\partial\Omega_i}\frac{\abs{v(x)-v(y)}^2}{\abs{x-y}^{d}}\,\mathrm{d}x\,\mathrm{d}y
	                        +\|v\|^2_{L^2(\partial\Omega_i)}\Big)^{1/2}.
	\end{aligned}
\end{equation}
Denoting the extension by zero from $\Gamma$ to $\partial\Omega_i$ by $\hat{E}_i$ we define the Lions--Magenes space as
\begin{displaymath}
	    \Lambda=\{\mu\in L^2(\Gamma): \hat{E}_i\mu\in H^{1/2}(\partial\Omega_i)\}\quad\text{with}\quad 
	    \|\mu\|_\Lambda=\|\hat{E}_i\mu\|_{H^{1/2}(\partial\Omega_i)}.
\end{displaymath}
Note that~\cite[Lemma A.8]{widlund} yields the identification $\Lambda\cong[H^{1/2}_0(\Gamma), L^2(\Gamma)]_{1/2}$, which explains why $\Lambda$ is independent of $i=1,2$. 

Since $H^{1/2}(\partial\Omega_i)$ is a separable Hilbert space, so is  $\Lambda$. On $V_i$ the trace operator takes the form
\begin{displaymath}
	     \hat{T}_i: V_i\rightarrow \Lambda:v\mapsto\restr{(\hat{T}_{\partial\Omega_i}v)}{\Gamma},
\end{displaymath}	     
and is bounded; see~\cite[Lemma 4.4]{EHEE22}.

For the temporal function space $H^s(\R)$, $s\in [0,1]$, we use the Fourier definition
\begin{equation}\label{eq:sobolevfourier}
	\begin{aligned}
            &H^s(\R)=\{v\in L^2(\R): (1+(\cdot)^2)^{s/2}\hat{\F} v\in L_\C^2(\R)\}\quad\text{with}\\
             &\|v\|_{H^s(\R)}=\|(1+(\cdot)^2)^{s/2}\hat{\F} v\|_{L_\C^2(\R)}.
         \end{aligned}
\end{equation}
Here, $\hat{\F}$ is the Fourier transform and $L_\C^2(\R)$ is the complexification of the real Hilbert space $L^2(\R)$. Note that this is equivalent to the Sobolev--Slobodetskii definition \cref{eq:slobodetskii} for $s=1/2$ and $\partial\Omega_{i}=\R$; see~\cite[Lemma 16.3]{tartar}. We then introduce the temporal Hilbert transform by
\begin{displaymath}
        \hat{\Hc} v(t)=\lim_{\epsilon\rightarrow 0^{+}}\frac{1}{\pi}\int_{\abs{s}\geq\epsilon}\frac{1}{s}v(t-s)\,\mathrm{d}s.
\end{displaymath}
\noindent
From \cite[Chapters 4, 5]{king} we have that $\hat{\Hc}: L^2(\R)\rightarrow L^2(\R)$ is an isomorphism with inverse $\hat{\Hc}^{-1}=-\hat{\Hc}$ and
\begin{equation}\label{eq:hilbertfourier}
        \hat{\Hc}=\hat{\F}^{-1}\hat{M}_\sgn\hat{\F},
\end{equation}
where $\hat{M}_\sgn v(\xi)=-\mathrm{i}\,\sgn(\xi)v(\xi)$. The formula \cref{eq:hilbertfourier} combined with the definition \cref{eq:sobolevfourier} shows that $\hat{\Hc}: H^s(\R)\rightarrow H^s(\R)$ is an isomorphism. We also introduce the temporal half-derivatives as
\begin{equation}\label{eq:halftime}
\hat{\partial}^{1/2}_\pm=\hat{\F}^{-1}\hat{M}_{\pm}\hat{\F},
\end{equation}
where $\hat{M}_+v(\xi)=\sqrt{\mathrm{i}\xi} v(\xi)$ and $\hat{M}_-v(\xi)=\overline{\sqrt{\mathrm{i}\xi}}v(\xi)$. It is clear from the definition \cref{eq:sobolevfourier} that $\hat{\partial}^{1/2}_\pm: H^{1/2}(\R)\rightarrow L^2(\R)$ are bounded linear operators. The important relations between these operators are given in the lemma below.
\begin{lemma}\label{lemma:timeder}
For $v\in H^{1/2}(\R)$ one has the equalities 
\begin{displaymath}
\hat{\partial}^{1/2}_+ v = -\hat{\partial}^{1/2}_-\hat{\Hc} v\quad\text{and}\quad (\hat{\partial}^{1/2}_+ v\,,\, \hat{\partial}^{1/2}_- v)_{L^2(\R)}=0.
\end{displaymath}
Moreover, for $v\in H^1(\R)$ and $w\in H^{1/2}(\R)$ one has the fractional integration by parts formula
\begin{displaymath}
(\hat{\partial}_t v\,,\, w)_{L^2(\R)}=(\hat{\partial}^{1/2}_+ v\,,\, \hat{\partial}^{1/2}_- w)_{L^2(\R)}.
\end{displaymath}
\end{lemma}
\begin{proof}
Let $v\in H^{1/2}(\R)$ and observe that
\begin{displaymath}
\sqrt{\mathrm{i}\xi}=-\overline{\sqrt{\mathrm{i}\xi}}(-\mathrm{i}\,\sgn(\xi)). 
\end{displaymath}
The Fourier characterization of the operators $\hat{\Hc},\hat{\partial}^{1/2}_+, \hat{\partial}^{1/2}_-$ then implies that 
\begin{displaymath}
\hat{\partial}^{1/2}_+v =-\hat{\partial}^{1/2}_-\hat{\Hc}v. 
\end{displaymath}
A similar argument, together with the fact that $|\hat{\F}v|^2$ is an even function, shows that
\begin{displaymath}
(\hat{\partial}^{1/2}_+ v, \hat{\partial}^{1/2}_- v)_{L^2(\R)}=\int_\R i\xi|\hat{\F}v|^2\mathrm{d}\xi=0
\end{displaymath}
for $v\in H^{1/2}(\R)$. Finally, the fractional integration by parts formula follows by the Fourier characterization $\hat{\partial}_t=\hat{\F}^{-1}\hat{M}\hat{\F}$, where $\hat{M}v(\xi)=\mathrm{i} \xi v(\xi)$.
\end{proof}
The Fourier characterization of the operator $\hat{\partial}^{1/2}_+$ and~\cref{eq:sobolevfourier} yield that
\begin{displaymath}
v\mapsto\bigl(\|\hat{\partial}^{1/2}_+v\|^2_{L^2(\R)}+\|v\|^2_{L^2(\R)}\bigr)^{1/2}
\end{displaymath}
is an equivalent norm on $H^{1/2}(\R)$.

\section{Tensor spaces}\label{sec:tensor}
Inspired by the finite element analysis in~\cite{steinbach20}, we identify our Bochner spaces in space-time as tensor spaces. A general introduction to tensor spaces can be found in ~\cite[Chapter 3.4]{weidmann}. 

We denote the algebraic tensor product of two (real) separable Hilbert spaces $X, Y$ by $X\otimes Y$. For elements of the form $x\otimes y$ the inner product is defined as
\begin{displaymath}
(x_1\otimes y_1, x_2\otimes y_2)_{X\otimes Y}=(x_1, x_2)_X(y_1, y_2)_Y
\end{displaymath}
and for arbitrary elements in $X\otimes Y$ the definition is extended by linearity. The closure of $X\otimes Y$ with respect to the induced norm is denoted by $X\tensor Y$. From these definitions it follows that  $X\tensor Y = Y\tensor X$. 

If $\{x_k\}_{k\geq 0}$ and $\{y_\ell\}_{\ell\geq 0}$ are orthonormal bases of $X$ and $Y$, respectively, then\linebreak $\{x_k\otimes y_\ell\}_{k,\ell\geq 0}$ is an orthonormal basis of $X\tensor Y$ and every $v\in X\tensor Y$ can be represented as 
\begin{equation}\label{eq:uiorth}
v=\sum_{k,\ell=0}^\infty c_{k,\ell}(x_k \otimes y_\ell)=\sum_{k=0}^\infty x_k \otimes z_k ,\quad z_k=\sum_{\ell=0}^\infty c_{k,\ell}y_\ell\in Y,
\end{equation} 
where $\{c_{k,\ell}\}_{k,\ell\geq 0}$ are real coefficients. This follows from~\cite[Theorem 3.12]{weidmann}. 

We recall the following result on the extensions of operators to tensor spaces. The proof of \cref{lemma:tensor} can be found in~\cite[Section 12.4.1]{aubin}. 
\begin{lemma}\label{lemma:tensor}
Let $X_{k}, Y_{k}$, $k=1,2$, be separable Hilbert spaces and $A: X_1\rightarrow X_2$, $B: Y_1\rightarrow Y_2$ be bounded linear operators. Then there is a bounded linear operator
\begin{displaymath}
                A\tensor B: X_1\tensor Y_1\rightarrow X_2\tensor Y_2
\end{displaymath}
 such that $(A\tensor B)(x\otimes y)=Ax\otimes By$ for every $x\in X_1, y\in Y_1$.
\end{lemma}
From \cref{lemma:tensor} it follows that the spatial trace operators
\begin{align*}
            T_{\partial\Omega_i}&=I\tensor\hat{T}_{\partial\Omega_i}:L^2(\R)\tensor H^1(\Omega)\rightarrow L^2(\R)\tensor H^{1/2}(\partial\Omega_i), \\
            T_i&=I\tensor\hat{T}_i:L^2(\R)\tensor V_i\rightarrow L^2(\R)\tensor \Lambda
\end{align*}
are bounded. Furthermore, we have the identity
\begin{displaymath}
T_iv=\restr{(T_{\partial\Omega_i}v)}{\Gamma\times\R}\quad\text{for }v\in L^2(\R)\tensor V_i. 
\end{displaymath}
This is easily proven by validating the identity on the dense subset $L^2(\R)\otimes V_i$ and using the continuity of the operators. Note that the restrictions, as well as the extension by zero $E_i=I\tensor\hat{E}_i$, are all well defined operations due to~\cref{lemma:tensor}. 

For any separable Hilbert space $X$ one has the identification
\begin{displaymath}
H^s(\R)\tensor X\cong H^s(\R, X),
\end{displaymath}
where $H^s(\R, X)$, $s\in [0,1]$, is a Sobolev--Bochner space; see~\cite[Chapter 2.5.d]{hytonen}. The identification can be proven by noting the following facts. The norms coincide on $H^s(\R)\otimes X$, one has the relation $H^1(\R)\tensor X\cong H^1(\R, X)$, and the spaces $H^1(\R)\otimes X$ and $H^1(\R, X)$ are dense in $H^s(\R)\tensor X$ and $H^s(\R, X)$, respectively. For proofs see~\cite[Theorem 12.7.1]{aubin}, \cite[Theorem 3.12]{weidmann}, and~\cite[Proposition 6.1]{lionsmagenes1}.
 
 The Sobolev--Bochner spaces below will make up the core of the analysis:
\begin{align*}
            W&=H^{1/2}(\R)\tensor L^2(\Omega)\,\cap\, L^2(\R)\tensor V, \\
	    W_i^0&=H^{1/2}(\R)\tensor L^2(\Omega_i)\,\cap\, L^2(\R)\tensor V_i^0, \\
	    W_i&=H^{1/2}(\R)\tensor L^2(\Omega_i)\,\cap\, L^2(\R)\tensor V_i, \\
	     Z&=H^{1/4}(\R)\tensor L^2(\Gamma)\,\cap\, L^2(\R)\tensor \Lambda.
\end{align*}
\begin{lemma}\label{lemma:identities}
Let \cref{ass:domains} be valid. Then we have the identities
\begin{align}
            L^2(\R)\tensor V_i^0&=\{v\in L^2(\R)\tensor H^1(\Omega_i): T_{\partial\Omega_i}v=0\},\label{eq:Wi0}\\
            L^2(\R)\tensor V_i&=\{v\in L^2(\R)\tensor H^1(\Omega_i): \restr{(T_{\partial\Omega_i}v)}{(\partial\Omega_i\setminus\Gamma)\times\R}=0\},\label{eq:Wi}\\
            L^2(\R)\tensor \Lambda&=\{\mu\in L^2(\R)\tensor L^2(\Gamma): E_i\mu\in L^2(\R)\tensor H^{1/2}(\partial\Omega_i)\}.\label{eq:Z}
\end{align}
\end{lemma}
\begin{proof}
Throughout the proof let $\{x_k\}_{k\geq 0}$ be an orthonormal basis of $L^2(\R)$. Let $\{y_\ell\}_{\ell\geq 0}$ be an orthonormal basis of a separable Hilbert space $Y$ given by the context, and the corresponding $\{z_k\}_{k\geq 0}$ are defined as in~\cref{eq:uiorth}.

To prove the identity~\cref{eq:Wi0} recall that $V_i^0=\{y\in H^1(\Omega_i): \hat{T}_{\partial\Omega_i}y=0\}$; see~\cite[Theorem~6.6.4]{kufner}. 
This observation together with the continuity of $T_{\partial\Omega_i}$ then yields that every $v\in L^2(\R)\tensor V_i^0$ satisfies 
\begin{displaymath}
T_{\partial\Omega_i}v=\sum_{k,\ell=0}^\infty c_{k,\ell}(x_k \otimes \hat{T}_{\partial\Omega_i}y_\ell)=0.
\end{displaymath}
Conversely, suppose that $v\in L^2(\R)\tensor H^1(\Omega_i)$ with $T_{\partial\Omega_i}v=0$. By the continuity of $T_{\partial\Omega_i}$ we have 
\begin{displaymath}
0=T_{\partial\Omega_i}v=\sum_{k=0}^\infty x_k \otimes \hat{T}_{\partial\Omega_i}z_k.
\end{displaymath}
It then follows by the orthonormality of $\{x_k\}_{k\geq 0}$ that 
\begin{equation}\label{eq:ort}
0 = (T_{\partial\Omega_i}v \, ,\, x_k\otimes \hat{T}_{\partial\Omega_i}z_k)_{L^2(\R)\tensor L^{2}(\partial\Omega_i)}=\|\hat{T}_{\partial\Omega_i}z_k\|^2_{L^2(\partial\Omega_i)}.
\end{equation}
Hence, $z_k\in V_i^0$, $k=0,1,\dots$, and from~\cref{eq:uiorth} we can conclude that $v\in L^2(\R)\tensor V_i^0$.

The proof of~\cref{eq:Wi} is similar after observing the definition $V_i=\{v\in H^1(\Omega_i): \restr{(\hat{T}_{\partial\Omega_i}v)}{\partial\Omega_i\setminus\Gamma}=0\}$ and employing the $L^2(\R)\tensor L^{2}(\partial\Omega_i\setminus\Gamma)$ inner product in~\cref{eq:ort} instead of $L^2(\R)\tensor L^{2}(\partial\Omega_i)$.

To prove~\cref{eq:Z} first observe that $E_i$ can be interpreted as a map from $L^2(\R)\tensor \Lambda$ into $L^2(\R)\tensor H^{1/2}(\partial\Omega_i)$, by~\cref{lemma:tensor}. This immediately gives the inclusion from left to right. Conversely, assume that $\mu\in L^2(\R)\tensor L^2(\Gamma)$ with $E_i\mu\in L^2(\R)\tensor H^{1/2}(\partial\Omega_i)$. By~\cref{eq:uiorth}, we have the representation
\begin{displaymath}
E_i\mu=\sum_{k=0}^\infty x_k \otimes z_k,\quad z_k\in H^{1/2}(\partial\Omega_i).
\end{displaymath}
As $\restr{(E_i\mu)}{(\partial\Omega_i\setminus\Gamma)\times\R}=0$, the orthonormality of $\{x_k\}_{k\geq 0}$ yields that
\begin{displaymath}
0 = (E_i\mu \, ,\, x_k\otimes z_k)_{L^2(\R)\tensor L^{2}(\partial\Omega_i\setminus\Gamma)}=\|z_k\|^2_{L^2(\partial\Omega_i\setminus\Gamma)},
\end{displaymath}
i.e., $\restr{z_k}{\partial\Omega_i\setminus\Gamma}=0$. As $\hat{E}_i$ is an isometry from $\Lambda$ onto $\{\mu\in H^{1/2}(\partial\Omega_i): \restr{\mu}{\partial\Omega_i\setminus\Gamma}=0\}$; compare with~\cite[Lemma~4.1]{EHEE22}, we obtain that $\restr{z_k}{\Gamma}\in\Lambda$ for all $k=0,1,\dots$ Since
\begin{displaymath}
\mu=\restr{(E_i\mu)}{\Gamma\times\R}=\sum_{k=0}^\infty x_k \otimes \restr{z_k}{\Gamma}\quad\text{in }L^2(\R)\tensor L^2(\Gamma)
\end{displaymath}
and $\{\sum_{k=0}^n x_k \otimes \restr{z_k}{\Gamma}\}_{n\geq0}$ is a Cauchy sequence in $L^2(\R)\tensor \Lambda$, we conclude that\linebreak $\mu\in L^2(\R)\tensor \Lambda$.
\end{proof}

\begin{lemma}\label{lemma:density}
If \cref{ass:domains} holds, then $Z$ is dense in $L^2(\Gamma\times\R)$.
\end{lemma}
\begin{proof}
The spaces $H^{1/4}(\R)$ and $\Lambda$ are dense in $L^2(\R)$ and $L^2(\Gamma)$, respectively; see~\cite[Lemma 15.10]{tartar} and \cite[Lemma 4.2]{EHEE22}. Therefore, by \cite[Theorem 3.12]{weidmann}, the corresponding algebraic tensor space $H^{1/4}(\R)\otimes \Lambda$ is dense in $L^2(\R)\tensor L^2(\Gamma)\cong L^2(\Gamma\times\R)$. The density of $Z$ in $L^2(\Gamma\times\R)$ then follows as $H^{1/4}(\R)\otimes\Lambda\subset Z$.
\end{proof}
\begin{lemma}\label{lemma:tracei}
If \cref{ass:domains} holds, then  $T_i:W_i\rightarrow Z$ is bounded and has a linear bounded right inverse $R_i:Z\rightarrow W_i$.
\end{lemma}
\begin{proof}
It follows from the density of $C^\infty_0(\R)\otimes C^\infty(\Bar{\Omega}_i)$ in $L^2(\R)\tensor H^1(\Omega_i)$; see \cite[Theorem 3.12]{weidmann}, that our definition of $T_{\partial\Omega_i}$ coincide with the definition given in~\cite[Lemma 2.4]{costabel90}. Hence, the restricted trace operator
\begin{displaymath}
 T_{\partial\Omega_i}:H^{1/2}(\R)\tensor L^2(\Omega_i)\cap L^2(\R)\tensor H^1(\Omega_i)\rightarrow
                                   H^{1/4}(\R)\tensor L^2(\Omega_i)\cap L^2(\R)\tensor H^{1/2}(\partial\Omega_i)
\end{displaymath}
is well defined and bounded. 

If $v\in W_i$ then $T_iv=\restr{(T_{\partial\Omega_i}v)}{\Gamma\times\R}\in H^{1/4}(\R)\tensor L^2(\Gamma)$. As $W_i\subset L^2(\R)\tensor V_i $, we have by definition that $T_iv\in L^2(\R)\tensor \Lambda$, i.e., $T_i:W_i\rightarrow Z$. The boundedness of $T_i:W_i\rightarrow Z$ then follows as the operators $T_{\partial\Omega_i}:H^{1/2}(\R)\tensor L^2(\Omega_i)\cap L^2(\R)\tensor H^1(\Omega_i)\rightarrow H^{1/4}(\R)\tensor L^2(\Omega_i)$ and $T_i:L^2(\R)\tensor V_i\rightarrow L^2(\R)\tensor \Lambda$ are bounded.

We can explicitly construct a linear right inverse to $T_i$ by considering the linear heat equation, with $(\alpha,\beta,f)=(\nabla u,0,0)$. More precisely,  for any $\mu\in H^{1/4}(\R)\tensor L^2(\partial\Omega_i)\cap L^2(\R)\tensor H^{1/2}(\partial\Omega_i)$, there exists a weak solution $u_i\in H^{1/2}(\R)\tensor L^2(\Omega_i)\cap L^2(\R)\tensor H^1(\Omega_i)$ to the heat equation with $T_{\partial\Omega_i}u_i=\mu$. This follows by~\cite[Theorem 2.9 and Remark 2.10]{costabel90}. Furthermore, according to \cite[p. 515]{costabel90}, one has the bound  
\begin{displaymath}
    \|u_i\|_{H^{1/2}(\R)\tensor L^2(\Omega_i)\cap L^2(\R)\tensor H^1(\Omega_i)}\leq C\|\mu\|_{H^{1/4}(\R)\tensor L^2(\partial\Omega_i)\cap L^2(\R)\tensor H^{1/2}(\partial\Omega_i)}.
\end{displaymath}
Hence, $R_{\partial\Omega_i}:\mu\mapsto u_i$ is a linear bounded right inverse to $T_{\partial\Omega_i}$. If $\eta\in Z$, then by~\cref{eq:Wi,eq:Z} we obtain that $R_i=R_{\partial\Omega_i}E_i:Z\rightarrow W_i$ is a bounded linear right inverse to $T_i$.
\end{proof}

\begin{lemma}\label{lemma:abstracthilbert}
Let $X$ be a separable Hilbert space and $s\in [0,1]$. The Hilbert transform 
\begin{displaymath}
\Hc=\hat{\Hc}\tensor I:H^s(\R)\tensor X\rightarrow H^s(\R)\tensor X
\end{displaymath} 
is then an isomorphism.
\end{lemma}
\begin{proof}
According to \cref{lemma:tensor} the operator $\hat{\Hc}\otimes I$ extends to a bounded linear operator $\Hc: H^s(\R)\tensor X\rightarrow H^s(\R)\tensor X$. From \cref{eq:hilbertfourier} it follows that $\Hc:H^s(\R)\tensor X\rightarrow H^s(\R)\tensor X$ is an isomorphism.
\end{proof}

\begin{lemma}\label{lemma:hilbert}
If \cref{ass:domains} holds then the restricted operators
\begin{displaymath} 
\Hc_i:W_i\rightarrow W_i\quad\text{and}\quad\Hc_\Gamma:Z\rightarrow Z
\end{displaymath} 
are isomorphisms and satisfy $T_i\Hc_iv=\Hc_\Gamma T_iv$ for all $v\in W_i$.
\end{lemma}
\begin{proof}
By different choices of $s\in [0,1]$ and separable Hilbert spaces $X$ in~\cref{lemma:abstracthilbert} one obtains the isomorphisms
\begin{align*}
        \Hc_i&:L^2(\Omega_i\times\R)\rightarrow L^2(\Omega_i\times\R)\\
        \Hc_i&:H^{1/2}(\R)\tensor L^2(\Omega_i)\rightarrow H^{1/2}(\R)\tensor L^2(\Omega_i)\\
        \Hc_i&:L^2(\R)\tensor V_i\rightarrow L^2(\R)\tensor V_i.
\end{align*}
Since the operators coincide on $W_i$ we have that $\Hc_i:W_i\rightarrow W_i$ is an isomorphism. A similar argument shows that $\Hc_\Gamma:Z\rightarrow Z$ is an isomorphism. For $v\in L^2(\R)\otimes V_i$ we have
\begin{displaymath} 
T_i\Hc_iv=(I\otimes \hat{T}_i)(\hat{\Hc}\otimes I_{L^2(\Omega_i)})v=
                \hat{\Hc}\otimes \hat{T}_iv=
                (\hat{\Hc}\otimes I_{L^2(\Gamma)})(I\otimes \hat{T}_i)v=
                \Hc_\Gamma T_iv,
\end{displaymath} 
and by density the identity holds for $v\in W_i\subset L^2(\R)\tensor V_i$.
\end{proof}
Finally, let $\varphi\in [0,\pi/2]$ be a parameter to be chosen later and define
\begin{displaymath} 
\Hc_i^\varphi=\cos{(\varphi)}I-\sin{(\varphi)}\Hc_i\quad\text{and}\quad\Hc_\Gamma^\varphi=\cos{(\varphi)}I-\sin{(\varphi)}\Hc_\Gamma.
\end{displaymath} 
It follows from \cref{lemma:hilbert} that $\Hc_i^\varphi:W_i\rightarrow W_i$ and $\Hc_\Gamma^\varphi:Z\rightarrow Z$ are isomorphisms and
\begin{equation}\label{eq:hilbertphi}
T_i\Hc_i^\varphi=\Hc_\Gamma^\varphi T_i.
\end{equation}
The restricted operators $\Hc^\varphi_i:W^0_i\rightarrow W^0_i$ and $\Hc^\varphi_{\Omega}:W\rightarrow W$ are also isomorphisms. 

\section{Weak space-time formulations}\label{sec:weak}
The aim is now to derive a variational framework in which we can state the weak forms of the quasilinear parabolic equation~\cref{eq:strong} and of the corresponding transmission problem. As a start, \cref{lemma:tensor} yields that the extensions of the spatial gradient
\begin{displaymath} 
\nabla: L^2(\R)\tensor H^1(\Omega_i)\rightarrow L^2(\Omega_i\times\R)^d
\end{displaymath} 
and the temporal half-derivatives
\begin{displaymath} 
\partial^{1/2}_\pm: H^{1/2}(\R)\tensor L^2(\Omega_i)\rightarrow L^2(\Omega_i\times\R)
\end{displaymath} 
are all bounded linear operators. Note that we leave out the dependence on $i=1,2$ on the above operators for the sake of readability.
From~\cref{ass:eq} and~\cite[Theorem 3.1]{lucchetti} we obtain that the functions $\alpha,\beta$ extend to well defined Nemytskii operators 
\begin{displaymath}
\alpha: L^2(\Omega_i\times\R)\times L^2(\Omega_i\times\R)^d\rightarrow L^2(\Omega_i\times\R)^d,\quad
\beta: L^2(\Omega_i\times\R)\times L^2(\Omega_i\times\R)^d\rightarrow L^2(\Omega_i\times\R).
\end{displaymath}
Hence, the forms $a:W\times W\rightarrow \R$ and $a_i: W_i\times W_i\rightarrow\R$, $i=1,2$, defined by the formulas
\begin{align*}
        a(u, v)&=\int_\R\int_{\Omega}\partial^{1/2}_+u\partial^{1/2}_-v+\alpha(x, u, \nabla u)\cdot\nabla v+\beta(x, u, \nabla u)v\,\,\mathrm{d}x\,\mathrm{d}t\quad\text{and}\\
        a_i(u_i, v_i)&=\int_\R\int_{\Omega_i}\partial^{1/2}_+u_i\partial^{1/2}_-v_i+\alpha(x, u_i, \nabla u_i)\cdot\nabla v_i+\beta(x, u_i, \nabla u_i)v_i\,\,\mathrm{d}x\,\mathrm{d}t,
\end{align*}
respectively, are all well defined. 

Let $f\in L^2(\Omega\times \R)$ and introduce $f_i=\restr{f}{\Omega_i\times\R}$. The weak, or variational, formulation of the equation~\cref{eq:strong} is to find $u\in W$ such that
\begin{equation}\label{eq:weak}
	    a(u, v)=(f, v)_{L^2(\Omega\times \R)}\quad \textrm{for all }v\in W.
\end{equation}
Here, the weak problem is derived by multiplying \cref{eq:strong} by $v\in W$, integrating, and using the fractional integration by parts formula from \cref{lemma:timeder} extended to the tensor setting. For the sake of completeness, we also note that the weak problem on $\Omega_i\times\R$ , with homogenous boundary conditions, is given by finding $u_i\in W_i^0$ such that
\begin{equation}\label{eq:weakai}
a_i(u_i, v_i)=(f_i, v_i)_{L^2(\Omega\times \R)}\quad \textrm{for all } v_i\in W_i^0.
\end{equation}

The following result is based on the argumentation in~\cite[Section 2.8]{larsson15}, which is a summary of the monotone-equivalency idea from~\cite{fontesthesis}. 
\begin{lemma}\label{lemma:ai}
Let~\cref{ass:domains,ass:eq,ass:f} be valid. Then the form $a_i:W_i\times W_i\rightarrow\R$ is Lipschitz continuous and the form $a_i(\cdot, \Hc_i^\varphi\cdot):W_i\times W_i\rightarrow\R$ is uniformly monotone for a sufficiently small $\varphi>0$. Moreover, $a_i$ is also uniformly monotone in $L^2(\R)\tensor H^1(\Omega_i)$. An analogous result holds for~$a$. In particular there exists unique solutions to \cref{eq:weak,eq:weakai}, respectively.
\end{lemma}
\begin{proof}
We will only consider the case with the bilinear form~$a_i$, as the same proof holds for~$a$. First, note that~\cref{ass:eq} yields the bound 
\begin{displaymath}
    \|\alpha(u, \nabla u)-\alpha(v, \nabla v)\|_{L^2(\Omega_i\times\R)^d}
    \leq C\bigl(\|\nabla (u-v)\|_{L^2(\Omega_i\times\R)^d}+\|u-v\|_{L^2(\Omega_i\times\R)}\bigr)
\end{displaymath}
for every $u,v\in L^2(\R)\tensor H^1(\Omega_i)$. Hence, the Lipschitz continuity of $a_i$ on $W_i$ follows directly by the Cauchy--Schwarz inequality together with the boundedness of~$\nabla$ and $\partial^{1/2}_\pm$.

Before we adress the monotonicity bounds we first make a few observations. First, the identities in~\cref{lemma:timeder} can trivially be validated on $H^{1/2}(\R)\otimes L^2(\Omega_i)$ and extended to $H^{1/2}(\R)\tensor L^2(\Omega_i)$ by density and the boundedness of the related operators. That is, 
\begin{equation}\label{eq:exttime}
\partial^{1/2}_+ v = -\partial^{1/2}_-\Hc_i v\quad\text{and}\quad(\partial^{1/2}_+ v\, ,\, \partial^{1/2}_- v)_{L^2(\Omega_i\times\R)}=0
\end{equation} 
for every $v\in H^{1/2}(\R)\tensor L^2(\Omega_i)$. Second, as~\cref{ass:domains} is valid, a similar density argument together with the observations regarding the equivalent norms in~\cref{sec:prel} yields the ``extended'' Poincaré's inequality
\begin{displaymath}
\|v\|_{L^2(\Omega_i\times\R)}\leq C_p\|\nabla v\|_{L^2(\Omega_i\times \R)^d}\quad\text{for all }v\in L^2(\R)\tensor V_i,
\end{displaymath}
and the fact that 
\begin{equation}\label{eq:eqivnorms}
v\mapsto\|\nabla v\|_{L^2(\Omega_i\times\R)^d}\quad\text{and}\quad v\mapsto\bigl(\|\partial^{1/2}_+v\|^2_{L^2(\Omega_i\times\R)}+\|\nabla v\|^2_{L^2(\Omega_i\times\R)^d}\bigr)^{1/2}
\end{equation}
are (equivalent) norms on $L^2(\R)\tensor V_i$ and $W_i$, respectively. Third, by~\cref{lemma:abstracthilbert}, $\nabla\Hc_i$ is a bounded operator from $L^2(\R)\tensor H^1(\Omega_i)$ into $L^2(\Omega_i\times\R)^d$. 

In order to prove the monotonicity bounds we first consider the temporal term. By~\cref{eq:exttime} we have the equality
\begin{equation}\label{eq:tmon}
\begin{aligned}
    \bigl(&\partial^{1/2}_+ u\, ,\, \partial^{1/2}_-\Hc_i^\varphi (u-v)\bigr)_{L^2(\Omega_i\times\R)}-\bigl(\partial^{1/2}_+ v\, ,\, \partial^{1/2}_- \Hc_i^\varphi(u-v)\bigr)_{L^2(\Omega_i\times\R)}\\
    &=\bigl(\partial^{1/2}_+ (u-v)\, ,\, \partial^{1/2}_-\Hc_i^\varphi (u-v)\bigr)_{L^2(\Omega_i\times\R)}\\
    &=\cos(\varphi)\bigl(\partial^{1/2}_+ (u-v)\, ,\, \partial^{1/2}_- (u-v)\bigr)_{L^2(\Omega_i\times\R)}+\sin(\varphi)\bigl(\partial^{1/2}_+ (u-v)\, ,\, -\partial^{1/2}_-\Hc_i (u-v)\bigr)_{L^2(\Omega_i\times\R)}\\
    &=\sin(\varphi)\|\partial^{1/2}_+(u-v)\|^2_{L^2(\Omega_i\times\R)}
 \end{aligned}
\end{equation}
for every $u,v\in H^{1/2}(\R)\tensor L^2(\Omega_i)$.

Next, we consider the monotonicity bound for the spatial terms. Via~\cref{ass:eq} and the Poincaré's inequality we obtain that the Nemytskii operators $\alpha, \beta$ satisfy the monotonicity bound
\begin{align*}
    \int_\R\int_{\Omega_i}&\bigl(\alpha(u, \nabla u)-\alpha(v, \nabla v)\bigr)\cdot\nabla (u-v)
        +\bigl(\beta(u, \nabla u)-\beta(v, \nabla v)\bigr)(u-v)\,\mathrm{d}x\,\mathrm{d}t\\
        &\geq {\inf}_x h_2\, \|\nabla(u-v)\|^2_{L^2(\Omega_i\times\R)} - {\sup}_x h_3\,\|u-v\|^2_{L^2(\Omega_i\times\R)}\geq c \|\nabla(u-v)\|^2_{L^2(\Omega_i\times\R)}
\end{align*}
for every $u,v\in L^2(\R)\tensor V_i$. 

Making use of the Lipschitz continuity of $\alpha, \beta$ and the previous monotonicity bound gives us the inequality
\begin{equation}\label{eq:smon}
\begin{aligned}
        \int_\R&\int_{\Omega_i}\bigl(\alpha(u, \nabla u)-\alpha(v, \nabla v)\bigr)\cdot\nabla \Hc_i^\varphi(u-v)
         +\bigl(\beta(u, \nabla u)-\beta(v, \nabla v)\bigr)\Hc_i^\varphi(u-v)\,\mathrm{d}x\,\mathrm{d}t\\
        &=\cos(\varphi)\int_\R\int_{\Omega_i}\bigl(\alpha(u, \nabla u)-\alpha(v, \nabla v)\bigr)\cdot\nabla (u-v)
           +\bigl(\beta(u, \nabla u)-\beta(v, \nabla v)\bigr)(u-v)\,\mathrm{d}x\,\mathrm{d}t\\
        &-\sin(\varphi)\int_\R\int_{\Omega_i}\bigl(\alpha(u, \nabla u)-\alpha(v, \nabla v)\bigr)\cdot\nabla \Hc_i(u-v)
        +\bigl(\beta(u, \nabla u)-\beta(v, \nabla v)\bigr)\Hc_i(u-v\bigr)\,\mathrm{d}x\,\mathrm{d}t\\
        &\geq c\cos(\varphi)\|\nabla(u-v)\|^2_{L^2(\Omega_i\times\R)}-C|\sin(\varphi)|\|\nabla(u-v)\|^2_{L^2(\Omega_i\times\R)}\\
        &= \bigl(c\cos(\varphi)-C|\sin(\varphi)|\bigr)\|\nabla(u-v)\|^2_{L^2(\Omega_i\times\R)}
\end{aligned}
\end{equation}
for every $u,v\in L^2(\R)\tensor V_i$.

Summing the bounds~\cref{eq:tmon,eq:smon} and choosing $\varphi>0$ small enough yields
\begin{align*} 
        a_i\bigl(&u, \Hc_i^\varphi(u-v)\bigr)-a_i\bigl(v, \Hc_i^\varphi(u-v)\bigr)\\
                 &\geq  \sin(\varphi)\|\partial^{1/2}_+(u-v)\|^2_{L^2(\R)\tensor L^2(\Omega_i)}
                    +\bigl(c\cos(\varphi)-C\sin(\varphi)\bigr)\|\nabla(u-v)\|^2_{L^2(\Omega_i\times\R)}\\
                &\geq c\|u-v\|_{W_i}^2,
\end{align*}
for every $u,v\in W_i$. This proves that $a_i(\cdot, \Hc_i^\varphi\cdot)$ is uniformly monotone. Similarly, choosing $\varphi=0$ yields the uniform monotonicity in $L^2(\R)\tensor H^1(\Omega_i)$, i.e.,
\begin{align*} 
        a_i(u, u-v)-a_i(v, u-v)\geq c\|\nabla(u-v)\|^2_{L^2(\Omega_i\times\R)}\geq c\|u-v\|_{L^2(\R)\tensor V_i}^2=c\|u-v\|_{L^2(\R)\tensor H^1(\Omega_i)}^2
\end{align*}
for every $u,v\in W_i$. The fact that~\cref{eq:weakai} has a unique solution now follows from~\cite[Theorem 25.B]{zeidler2b} applied to the operator $A_i:W_i\rightarrow W_i^*:u\mapsto a_i(u, \Hc_i^\varphi\cdot)$.
\end{proof}

Before moving on to the weak form of the transmission problem it is necessary to prove that we can ``glue'' together functions in our $H^{1/2}$-framework. This is the purpose of the following lemma.

\begin{lemma}\label{lemma:cutpaste}
Suppose that \cref{ass:domains} holds. Let $(v_1,v_2)\in W_1\times W_2$ and define $v=\{v_i\textrm{ on }\Omega_i\times\R,\, i=1,2\}$. If $T_1v_1=T_2v_2$ then $v\in W$. Conversely, let $v\in W$ and define $v_i=\restr{v}{\Omega_i\times\R}$. Then $T_iv_i\in Z$ and $T_1v_1=T_2v_2$.
\end{lemma}
\begin{proof}
Assume that $v_i\in W_i$, $i=1,2$, and $T_1v_1=T_2v_2$. Then 
\begin{displaymath}
v=\{v_i\textrm{ on }\Omega_i\times\R,\, i=1,2\}\in L^2(\Omega\times\R).
\end{displaymath}
In order to prove the $L^2(\R)\tensor V$-regularity of $v$, let $\{x_k\}_{k\geq 0}$ be an orthonormal basis of $L^2(\R)$ and $\{(y_i)_\ell\}_{\ell\geq 0}$ be orthonormal bases of $V_i$, $i=1,2$. The corresponding elements $\{(z_i)_k\}_{k\geq 0}\subset V_i$ are defined as in~\cref{eq:uiorth}. This yields the representation
\begin{displaymath}
v_i=\sum_{k=0}^\infty x_k \otimes (z_i)_k
\end{displaymath}
and the equality 
\begin{displaymath}
 0=(T_1v_1-T_2v_2\, ,\,x_k\otimes (\hat{T}(z_1)_k-\hat{T}_2(z_2)_k)_{L^2(\R)\tensor L^2(\Gamma)}=\|\hat{T}_1(z_1)_k-\hat{T}_2(z_2)_k\|^2_{ L^2(\Gamma)}.
\end{displaymath}
That is,  $\hat{T}_1(z_1)_k=\hat{T}_2(z_2)_k$, for all $k=0,1,\ldots$, and from \cite[Lemma 4.6]{EHEE22} it follows that $z_k=\{(z_i)_k \textrm{ on }\Omega_i\times\R,\, i=1,2\}\in V$. We also have the identification
\begin{displaymath}
v=\sum_{k=0}^\infty x_k \otimes z_k\quad\text{in }L^2(\Omega\times\R).
\end{displaymath}
As $\|z_k\|^2_{V}=\|(z_1)_k\|^2_{V_1}+\|(z_2)_k\|^2_{V_2}$, one has that $\{\sum_{k=0}^n x_k \otimes z_k\}_{n\geq0}$ is a Cauchy sequence in $L^2(\R)\tensor V$. Hence, $v$ is also an element in $L^2(\R)\tensor V$. 

The $H^{1/2}(\R)\tensor L^2(\Omega)$-regularity of $v$ follows in a similar fashion, by expanding the $v_i$ elements in terms of an orthonormal basis $\{x_k\}_{k\geq 0}$ of $H^{1/2}(\R)$ and  an orthonormal basis $\{(y_i)_k\}_{k\geq 0}$ of $L^2(\Omega_i)$ together with the observation that $\|z_k\|^2_{L^2(\Omega)}=\|(z_1)_k\|^2_{L^2(\Omega_1)}+\|(z_2)_k\|^2_{L^2(\Omega_2)}$. In conclusion, $v\in W$.

Conversely, let $v\in W$. From~\cref{eq:Wi} we get that $v_i=\restr{v}{\Omega_i\times\R}$ is an element in $W_i$ and, by~\cref{lemma:tracei}, $T_iv_i\in Z$. Let $\{x_k\}_{k\geq 0}$ be an orthonormal basis of $L^2(\R)$ and $\{y_\ell\}_{\ell\geq 0}$ an orthonormal basis of $V$. The related elements $\{z_k\}_{k\geq 0}\subset V$ are given by~\cref{eq:uiorth} and we obtain
\begin{displaymath}
v=\sum_{k=0}^\infty x_k \otimes z_k\quad\text{and}\quad v_i=\sum_{k=0}^\infty x_k \otimes (z_i)_k,
\end{displaymath}
where $(z_i)_k=\restr{z_k}{\Omega_i\times\R}\in V_i$. By~\cite[Lemma 4.5]{EHEE22} it follows that $\hat{T}_1(z_1)_k=\hat{T}_2(z_2)_k$ for every $k=0,1,\ldots$ Therefore
\begin{displaymath}
T_1v_1=\sum_{k,\ell=0}^\infty x_k \otimes \hat{T}_1(z_1)_k=\sum_{k,\ell=0}^\infty x_k \otimes \hat{T}_2(z_2)_k=T_2v_2,
\end{displaymath}
and the sought after equality is obtained. 
\end{proof}
The weak transmission problem is to find $(u_1, u_2)\in W_1\times W_2$ such that
\begin{equation}\label{eq:weaktran}
	\left\{\begin{aligned}
	     a_i(u_i, v_i)&=(f_i, v_i)_{L^2(\Omega_i\times\R)} & & \text{for all } v_i\in W_i^0,\, i=1,2,\\
	     T_1u_1&=T_2u_2, & &\\
	     \textstyle\sum_{i=1}^2 a_i(& u_i, R_i\mu)-(f_i, R_i\mu)_{L^2(\Omega_i\times\R)}=0 & &\text{for all }\mu\in Z. 
	\end{aligned}\right.
\end{equation}
\begin{lemma}\label{lemma:tranequiv}
Suppose that \cref{ass:domains,ass:eq,ass:f} hold. Then the weak equation is equivalent to the weak transmission problem in the following way: If $u$ solves \cref{eq:weak} then $(u_1, u_2)=(\restr{u}{ \Omega_1\times\R}, \restr{u}{ \Omega_2\times\R})$ solves \cref{eq:weaktran}. Conversely, if $(u_1, u_2)$ solves \cref{eq:weaktran} then $u=\{u_i \textrm{ on } \Omega_i\times\R\, , i=1,2\}$ solves \cref{eq:weak}. In particular, there exists a unique solution to \cref{eq:weaktran}.
\end{lemma}
\begin{remark}
\emph{The proof of \cref{lemma:tranequiv} follows by the same argument as~\cite[Lemma 1.2.1]{quarteroni} and requires that \cref{lemma:cutpaste} holds. This is one of the reason that the analysis is performed in the $H^{1/2}$-setting. As already stated in the introduction, analogous results to \cref{lemma:cutpaste} are not always true; see for instance~\cite[Example 2.14]{costabel90} for a counterexample in $H^1(\R)\tensor H^{-1}(\Omega)\cap L^2(\R)\tensor H^1(\Omega)$.}
\end{remark}

\section{Nonlinear time-dependent Steklov--Poincar\'e operators}\label{sec:SP}
The goal is now to reformulate the transmission problem to a problem on the interface $\Gamma\times\R$. In order to do so, one is required to consider non-homogeneous boundary values on the interface.
\begin{lemma}\label{lemma:Fi}
Let \cref{ass:domains,ass:eq,ass:f} be valid. For any $\eta\in Z$ there exists a unique $u_i\in W_i$ such that $T_iu_i=\eta$ and
\begin{displaymath}
a_i(u_i, v_i)=(f_i, v_i)_{L^2(\Omega_i\times\R)}\quad\text{for all }v_i\in W_i^0.
\end{displaymath}
\end{lemma}
\begin{proof}
Consider the shifted form $b_i: W_i^0\times W_i^0\rightarrow \R$ defined as
\begin{displaymath}
    b_i(u, v)=a_i(u+R_i\eta, \Hc_i^\varphi v).
\end{displaymath}
By the Lipschitz continuity of $a_i$ in~\cref{lemma:ai} we have that
\begin{align*}
    \bigl|b_i&(u, w)-b_i(v, w)\bigr|=\bigl|a_i(u+R_i\eta, \Hc_i^\varphi w)-a_i(v+R_i\eta, \Hc_i^\varphi w)\bigr|\\
    &\leq C\|(u+R_i\eta)-(v+R_i\eta)\|_{W_i}\|\Hc_i^\varphi w\|_{W_i}\leq C\|u-v\|_{W_i}\|w\|_{W_i},
\end{align*}
which shows that $b_i$ is also Lipschitz continuous. By the uniform monotonicity of $a_i(\cdot, \Hc_i^\varphi\cdot)$ in~\cref{lemma:ai} we have that
\begin{align*}
    b_i&(u, u-v)-b_i(v, u-v)=a_i\bigl(u+R_i\eta, \Hc_i^\varphi(u-v)\bigr)-a_i\bigl(v+R_i\eta, \Hc_i^\varphi (u-v)\bigr)\\
    &=a_i\bigl(u+R_i\eta, \Hc_i^\varphi(u+R_i\eta)-\Hc_i^\varphi(v+R_i\eta)\bigr)-a_i\bigl(v+R_i\eta, \Hc_i^\varphi (u+R_i\eta)-\Hc_i^\varphi(v+R_i\eta)\bigr)\\
    &\geq c\|(u+R_i\eta)-(v+R_i\eta)\|_{W_i}^2=c\|u-v\|_{W_i}^2,
\end{align*}
which shows that $b_i$ is uniformly monotone. Therefore, by~\cite[Theorem 25.B]{zeidler2b}, there exists a unique solution $u^0_i\in W_i^0$ to the problem
\begin{displaymath}
    b_i(u^0_i, v_i)=(f_i, \Hc_i^\varphi v_i)_{L^2(\Omega_i\times\R)}\quad\text{for all }v_i\in W_i^0.
\end{displaymath}
Defining $u_i=u^0_i+R_i\eta$ we have from~\cref{eq:Wi0} that $T_iu_i=\eta$. Moreover, since $\Hc_i^\varphi:W_i^0\rightarrow W_i^0$ is an isomorphism,
\begin{displaymath}
a_i(u_i, v_i)=(f_i,v_i)_{L^2(\Omega_i\times\R)}
\end{displaymath}
for all $v_i\in W_i^0$.
\end{proof}
According to \cref{lemma:Fi} there exists a (nonlinear) operator $F_i:Z\rightarrow W_i:\eta\mapsto u_i$ such that
\begin{equation}\label{eq:Fi}
	    a(F_i\eta, v_i)=(f_i,v_i)_{L^2(\Omega_i\times\R)}\quad \textrm{ for all }v_i\in W_i^0,
\end{equation}
with a (linear) left inverse $T_i$, i.e., $T_iF_i\eta=\eta$ for $\eta\in Z$. 
\begin{lemma}\label{lemma:Fi_lip}
Let \cref{ass:domains,ass:eq,ass:f} be valid. The operator $F_i:Z\rightarrow W_i$ is then Lipschitz continuous.
\end{lemma}
\begin{proof}
    Let $w_i=(\Hc_i^\varphi R_i\eta-\Hc_i^\varphi R_i\mu)-(\Hc_i^\varphi F_i\eta-\Hc_i^\varphi F_i\mu)$ and note that, by~\cref{eq:hilbertphi},
    \begin{align*}
        T_iw_i&=T_i(\Hc_i^\varphi R_i\eta-\Hc_i^\varphi R_i\mu)-T_i(\Hc_i^\varphi F_i\eta-\Hc_i^\varphi F_i\mu)\\
        &=(\Hc_\Gamma^\varphi \eta-\Hc_\Gamma^\varphi\mu)-(\Hc_\Gamma^\varphi\eta-\Hc_\Gamma^\varphi\mu)=0,
    \end{align*}
    which implies that $w_i\in W_i^0$ by~\cref{eq:Wi0}. Using~\cref{eq:Fi} together with~\cref{lemma:tracei,lemma:ai} we have that
    \begin{equation*}\label{eq:contbound1}
    \begin{aligned}
	        c\|F_i\eta-F_i\mu\|^2_{W_i}
	        &\leq a_i(F_i\eta, \Hc_i^\varphi F_i\eta-\Hc_i^\varphi F_i\mu)-a_i(F_i\mu, \Hc_i^\varphi F_i\eta-\Hc_i^\varphi F_i\mu)\\
	        &=a_i\bigl(F_i\eta, \Hc_i^\varphi R_i(\eta-\mu)\bigr)-a_i(F_i\eta,w_i)-a_i\bigl(F_i\mu, \Hc_i^\varphi R_i(\eta-\mu)\bigr)+a_i(F_i\mu,w_i)\\
	        &=a_i\bigl(F_i\eta, \Hc_i^\varphi R_i(\eta-\mu)\bigr)-\langle f_i, v_i\rangle-a_i\bigl(F_i\mu, \Hc_i^\varphi R_i(\eta-\mu)\bigr)+\langle f_i, v_i\rangle\\
	        &=a_i\bigl(F_i\eta, \Hc_i^\varphi R_i(\eta-\mu)\bigr)-a_i\bigl(F_i\mu, \Hc_i^\varphi R_i(\eta-\mu)\bigr)\\
	        &\leq C \|F_i\eta-F_i\mu\|_{W_i}\|\Hc_i^\varphi R_i(\eta-\mu)\|^2_{W_i}\leq C \|F_i\eta-F_i\mu\|_{W_i}\|\eta-\mu\|_Z.
    \end{aligned}
    \end{equation*}
    Dividing by $\|F_i\eta-F_i\mu\|_{W_i}$ proves the lemma.
\end{proof}
Next, we introduce the nonlinear time-dependent Steklov--Poincaré operators $S_i:Z\rightarrow Z^*$ by
\begin{displaymath}
\langle S_i\eta, \mu\rangle=a_i(F_i\eta, R_i\mu)-(f_i, R_i\mu)_{L^2(\Omega_i\times\R)}.
\end{displaymath}
We also write $S=S_1+S_2$. We can then introduce the weak Steklov--Poincaré equation $S\eta=0$ in $Z^*$ or, equivalently,
\begin{equation}\label{eq:speq}
	   \sum_{i=1}^2 \langle S_i\eta, \mu\rangle=0\quad \textrm{for all } \mu\in Z.
\end{equation}
\begin{remark}\label{rem:sp}
\emph{The Steklov--Poincaré operators do not depend on the choice of $R_i$. For an arbitrary extension $\tilde{R}_i:Z\rightarrow W_i$ such that $T_i\tilde{R}_i\mu=\mu$ we have, by \cref{eq:Wi0}, that $R_i\mu-\tilde{R}_i\mu\in W_i^0$. Combining this with \cref{eq:Fi} implies that
\begin{align*}
\langle S_i\eta, \mu\rangle&=a_i(F_i\eta, R_i\mu)-(f_i, R_i\mu)_{L^2(\Omega_i\times\R)}\\
&=a_i(F_i\eta, R_i\mu-\tilde{R}_i\mu)-(f_i, R_i\mu-\Tilde{R}_i\mu)_{L^2(\Omega_i\times\R)}
+a_i(F_i\eta, \tilde{R}_i\mu)-(f_i, \tilde{R}_i\mu)_{L^2(\Omega_i\times\R)}\\
&=a_i(F_i\eta, \tilde{R}_i\mu)-(f_i, \tilde{R}_i\mu)_{L^2(\Omega_i\times\R)}.
\end{align*}}
\end{remark}

The Steklov--Poincaré operators have similar properties as the forms $a_i$ in \cref{lemma:ai}.
\begin{lemma}\label{lemma:sp}
Suppose that \cref{ass:domains,ass:eq,ass:f} hold. Then the nonlinear time-dependent Steklov--Poincaré operator $S_i: Z\rightarrow Z^*$ is Lipschitz continuous, and uniform monotone in~$L^2(\R)\tensor\Lambda$. Moreover, for $\varphi>0$ small enough, the operator $(\Hc_\Gamma^\varphi)^*S_i: Z\rightarrow Z^*$ is uniformly monotone. Analogous results hold for $S$.
\end{lemma}
\begin{proof}
Throughout the proof, let $\eta,\mu,\lambda\in Z$ be arbitrary elements. The Lipschitz continuity of $S_i:Z\rightarrow Z^*$ is proved by \cref{lemma:ai,lemma:Fi_lip,lemma:tracei}, since
\begin{align*}
|\langle S_i\eta-S_i\mu, \lambda\rangle| &= |a_i(F_i\eta, R_i\lambda)-a_i(F_i\mu, R_i\lambda)|\leq C\|F_i\eta-F_i\mu\|_{W_i}\|R_i\lambda\|_{W_i}
\leq C\|\eta-\mu\|_Z\|\lambda\|_Z.
\end{align*}

To show the uniform monotonicity of $S_i$ in $L^2(\R)\tensor\Lambda$ let $w_i=R_i(\eta-\mu)-(F_i\eta-F_i\mu)$. Then
\begin{displaymath}
T_iw_i=T_i\bigl(R_i(\eta-\mu)-(F_i\eta-F_i\mu)\bigr)=0
\end{displaymath}
and therefore $w_i\in W_i^0$ by \cref{eq:Wi0}. This yields the monotonicity bound
\begin{align*}
\langle S_i\eta-S_i\mu, \eta-\mu\rangle&=a_i\bigl(F_i\eta, R_i(\eta-\mu))-a_i(F_i\mu, R_i(\eta-\mu)\bigr)\\
&=a_i(F_i\eta, F_i\eta-F_i\mu)-a_i(F_i\mu, F_i\eta-F_i\mu)+a_i(F_i\eta, w_i)-a_i(F_i\mu, w_i)\\
&\geq c\|F_i\eta-F_i\mu\|_{L^2(\R)\tensor H^1(\Omega_i)}^2\geq c\|T_i(F_i\eta-F_i\mu)\|_{L^2(\R)\tensor\Lambda}^2=c\|\eta-\mu\|_{L^2(\R)\tensor\Lambda}^2,
\end{align*}
using \cref{eq:Fi,lemma:ai} together with the fact that $T_i:L^2(\R)\tensor V_i\to L^2(\R)\tensor\Lambda$ is bounded.

In order to prove that the operator $(\Hc_\Gamma^\varphi)^*S_i: Z\rightarrow Z^*$ is uniformly monotone, we similarly introduce $w_i=R_i\Hc_\Gamma^\varphi(\eta-\mu)-\Hc_i^\varphi(F_i\eta-F_i\mu)$. From \cref{eq:hilbertphi} we have
\begin{align*}
T_iw_i&=T_i\bigl(R_i\Hc_\Gamma^\varphi(\eta-\mu)-\Hc_i^\varphi(F_i\eta-F_i\mu)\bigr)\\
&=\Hc_{\Gamma}^\varphi(\eta-\mu)-T_i\Hc_i^\varphi (F_i\eta-F_i\mu)\\
&=\Hc_{\Gamma}^\varphi(\eta-\mu)-\Hc_\Gamma^\varphi T_i(F_i\eta-F_i\mu)=0,
\end{align*}
and therefore $w_i\in W_i^0$ by \cref{eq:Wi0}. This together with~\cref{eq:Fi} implies that
\begin{align*}
 a_i\bigl(&F_i\eta, R_i\Hc_{\Gamma}^\varphi(\eta-\mu)\bigr)-a_i\bigl(F_i\mu, R_i\Hc_{\Gamma}^\varphi(\eta-\mu)\bigr) \\
 &=a_i\bigl(F_i\eta, \Hc_i^\varphi(F_i\eta-F_i\mu)\bigr)-a_i\bigl(F_i\mu, \Hc_i^\varphi(F_i\eta-F_i\mu)\bigr)
+a_i\bigl(F_i\eta, w_i\bigr)-a_i\bigl(F_i\mu, w_i\bigr)\\
 &=a_i\bigl(F_i\eta, \Hc_i^\varphi(F_i\eta-F_i\mu)\bigr)-a_i\bigl(F_i\mu, \Hc_i^\varphi(F_i\eta-F_i\mu)\bigr).
\end{align*}
Then \cref{lemma:tracei,lemma:ai} give,  for a sufficiently small parameter $\varphi>0$, that
\begin{align*}
\langle S_i\eta-S_i\mu, \Hc_\Gamma^\varphi(\eta-\mu)\rangle&=a_i\bigl(F_i\eta, \Hc_i^\varphi (F_i\eta-F_i\mu)\bigr)-a_i\bigl(F_i\mu, \Hc_i^\varphi (F_i\eta-F_i\mu)\bigr)\\
&\geq c\|F_i\eta-F_i\mu\|_{W_i}^2\geq c\|T_i(F_i\eta-F_i\mu)\|_Z^2=c\|\eta-\mu\|_Z^2.
\end{align*}
The bounds for $S$ follow by summing the bounds for $S_i$, $i=1,2$. 
\end{proof}
\begin{lemma}\label{lemma:tpspeq}
Suppose that \cref{ass:domains,ass:eq,ass:f} hold. The weak transmission problem and the weak Steklov--Poincaré equation are equivalent in the following way: If $(u_1, u_2)$ solves \cref{eq:weaktran} then $\eta=T_iu_i$ solves \cref{eq:speq}. Conversely, if $\eta$ solves \cref{eq:speq} then $(u_1, u_2)=(F_1\eta, F_2\eta)$ solves \cref{eq:weaktran}.
\end{lemma}
The proof of \cref{lemma:tpspeq} is immediate after writing out the definitions of the Steklov--Poincaré operators $S_i$, $i=1,2$. 

\section{Linear convergence of the modified Dirichlet--Neumann methods}\label{sec:lin}

The goal of this section is to develop new iterative methods for solving the weak Steklov--Poincaré equation $S\eta=0$ in $Z^*$; see~\cref{eq:speq}, that provably converges linearly (geometrically) without any extra regularity assumptions. As the $H^{1/2}$-framework presented in~\cref{sec:weak} resembles a nonlinear elliptic setting, especially with the uniform monotonicity property, it is natural to start off with a standard domain decomposition for elliptic problems. To this end, consider the Dirichlet--Neumann method 
\begin{equation}\label{eq:dn}
        \eta^{n+1}=\eta^n+s S_2^{-1}(0-S\eta^n),
\end{equation}
where $s>0$ is a method parameter. We refer to~\cite[Chapter~1.3]{quarteroni} for the derivation of~\cref{eq:dn}. The issue here is that the linear convergence analysis of the Dirichlet--Neumann method for elliptic equations relies on~$S_2$ being linear and symmetric. The latter is not valid as the time derivative in our parabolic problem is nonsymmetric. 

To resolve this, we approximate the solution to an equation of the form $G\eta=\chi$ in $X^*$ by the \emph{modified Dirichlet--Neumann} (MDN) method
\begin{equation}\label{eq:abstractmethod}
    \eta^{n+1}=\eta^n+s P^{-1}Q^*(\chi-G\eta^n),
\end{equation}
where $s>0$ is again a parameter$, \eta^0$ is a given initial guess, and the operators $P:X\rightarrow X^*$, $Q:X\rightarrow X$ are chosen appropriately. Here, the Dirichlet--Neumann method is recovered by setting
\begin{displaymath}
(X,G,Q,P)=(Z,S,I,S_2)\quad\text{and}\quad\chi=0.
\end{displaymath}
Before we derive our new methods we will prove a slight generalization of Zarantello's theorem~\cite[Theorem 25.B]{zeidler2b}. This generalization will characterize a problem/method family $(X,G,Q,P)$ that enables linear convergence. 
\begin{theorem}\label{thm:zarantello}
Let $X$ be a (real) Hilbert space and $G:X\rightarrow X^*$ be a nonlinear operator. Assume that there exists a linear isomorphism $Q:X\rightarrow X$
such that $Q^*G:X\rightarrow X^*$ is Lipschitz continuous and uniformly monotone. Furthermore, let $P:X\rightarrow X^*$ be any linear operator that is bounded, symmetric, and coercive.

Then $G$ is bijective, and for every $\chi\in X^*$, $\eta^0\in X$, and a sufficiently small $s>0$ the MDN iteration~\cref{eq:abstractmethod} converges to~$\eta$, the solution of
\begin{displaymath}
        G\eta=\chi.
\end{displaymath}
The converges is linear, i.e., there exists constants $C>0$ and $L\in(0,1)$ such that
\begin{displaymath}
	   \|\eta^n-\eta\|_X\leq C L^n\|\eta^0-\eta\|_X.
\end{displaymath}
\end{theorem}
\begin{proof}
Consider the operator $K\mu=\mu+s P^{-1}Q^*(\chi-G\mu)$. Then
\begin{displaymath}
        K\mu-K\lambda=\mu-\lambda-sP^{-1}(Q^*G\mu-Q^*G\lambda).
\end{displaymath}
We wish to show that $K:X\rightarrow X$ is a contraction. For this, we define the inner product
\begin{displaymath}
       (\mu, \lambda)_P=\langle P\mu, \lambda\rangle.
\end{displaymath}
It is clear that this defines a norm $\|\cdot\|_P$ that is equivalent to $\|\cdot\|_X$. Therefore, we will show that $K$ is a contraction in the norm $\|\cdot\|_P$. We split the norm into the three terms 
\begin{align*}
\|K\mu-K\lambda\|_P^2&=\bigl(\mu-\lambda-sP^{-1}(Q^*G\mu-Q^*G\lambda),\mu-\lambda-sP^{-1}(Q^*G\mu-Q^*G\lambda)\bigr)_P\\
        &=\|\mu-\lambda\|_P^2+s^2\bigl(P^{-1}(Q^*G\mu-Q^*G\lambda),P^{-1}(Q^*G\mu-Q^*G\lambda)\bigr)_P\\
        &\quad-s\Bigl(\bigl(\mu-\lambda,P^{-1}(Q^*G\mu-Q^*G\lambda)\bigr)_P+\bigl(P^{-1}(Q^*G\mu-Q^*G\lambda),\mu-\lambda\bigr)_P\Bigr)\\
        &=I_1+I_2+I_3.
\end{align*}
For the second term $I_2$ we use that $Q^*G$ is Lipschitz and $P^{-1}$ is bounded, which follows from the fact that $P$ is bounded and coercive. We also use the norm equivalence above to obtain
\begin{align*}
I_2&=\bigl(P^{-1}(Q^*G\mu-Q^*G\lambda),P^{-1}(Q^*G\mu-Q^*G\lambda)\bigr)_P=
	  \bigl\langle Q^*G\mu-Q^*G\lambda, P^{-1}(Q^*G\mu-Q^*G\lambda)\bigr\rangle\\
        &\leq C\|\mu-\lambda\|_X \|P^{-1}(Q^*G\mu-Q^*G\lambda)\|_X\leq C\|\mu-\lambda\|_X \|Q^*G\mu-Q^*G\lambda\|_{X^*}\\
        &\leq C\|\mu-\lambda\|_X \|\mu-\lambda\|_X\leq C\|\mu-\lambda\|_P^2.
	\end{align*}
For the third term $I_3$ we use the symmetry of $P$ and the uniform monotonicity of $Q^*G$. Again, we also use the norm equivalence above. These properties yield that
\begin{align*}
I_3&=-s\bigl(\mu-\lambda,P^{-1}(Q^*G\mu-Q^*G\lambda)\bigr)_P-s\bigl(P^{-1}(Q^*G\mu-Q^*G\lambda),\mu-\lambda\bigr)_P\\
        &=-s\bigl\langle P(\mu-\lambda),P^{-1}(Q^*G\mu-Q^*G\lambda)\bigr\rangle-s\bigl\langle Q^*G\mu-Q^*G\lambda,\mu-\lambda\bigr\rangle\\
        &=-2s\bigl\langle Q^*G\mu-Q^*G\lambda,\mu-\lambda\bigr\rangle\leq -cs\|\mu-\lambda\|_X^2\leq-cs\|\mu-\lambda\|_P^2.
\end{align*}
Thus we have that
\begin{displaymath}
	\|K\mu-K\lambda\|_P^2\leq(1+Cs^2-cs)\|\mu-\lambda\|_P^2.
\end{displaymath}
If we choose $s>0$ small enough then $K$ is a contraction and therefore there exists a unique fixed point $\eta\in X$ such that $P^{-1}Q^*(\chi-G\eta)=0$. Since $P, Q^*$ are both linear and bijective we have that $G\eta=\chi$. Finally, since $\chi\in X^*$ was arbitrary, we conclude that $G$ is bijective.

Now the error of the iteration~\cref{eq:abstractmethod} can be written as
 \begin{align*}
         \eta^{n+1}-\eta&=\eta^n-\eta+s P^{-1}Q^*(\chi-G\eta^n)\\
    		&=\eta^n-\eta+s P^{-1}(Q^*G\eta-Q^*G\eta^n)=K\eta^n-K\eta,
\end{align*}
and therefore
\begin{displaymath}
    		\|\eta^{n+1}-\eta\|_P\leq L\|\eta^n-\eta\|_P
\end{displaymath}
with $L=1+s^2C-cs$ as above. This, together with the norm equivalence, implies that
\begin{displaymath}
\|\eta^n-\eta\|_X\leq C\|\eta^n-\eta\|_P\leq C L^n\|\eta^0-\eta\|_P\leq C L^n\|\eta^0-\eta\|_X
\end{displaymath}
and the sought after linear convergence is obtained. 
\end{proof}
\begin{remark}
\emph{The proof of Zarantello's theorem in~\cite[Theorem 25.B]{zeidler2b} is given by the choice $Q=I$ and $P:u\mapsto (u,\cdot)_X$. We have already made use of this in~\cref{lemma:ai}, with 
\begin{displaymath}
(X,G,Q,P)=\bigl(W_i,\, u\mapsto a_i(u, \Hc_i^\varphi\cdot),\, I,\, u\mapsto (u,\cdot)_{W_i}\bigr),
\end{displaymath}
in order to prove existence and uniqueness of the corresponding parabolic problems.}
\end{remark}
As $\Hc_\Gamma^\varphi$ is a linear isomorphism on $Z$ and both $(\Hc_\Gamma^\varphi)^*S: Z\rightarrow Z^*$ and $(\Hc_\Gamma^\varphi)^*S_i: Z\rightarrow Z^*$ are Lipschitz continuous and uniformly monotone for a sufficiently small $\varphi>0$, according to~\cref{lemma:sp}, one directly obtains the result below from~\cref{thm:zarantello}. 
\begin{corollary}\label{cor:spbijective}
Let \cref{ass:domains,ass:eq,ass:f} hold. Then the Steklov--Poincaré operators $S, S_i:Z\rightarrow Z^*$ are bijective.
\end{corollary}
From~\cref{lemma:sp} it is also clear that our MDN methods should have the form
\begin{displaymath}
(X,G,Q,P)=(Z,S,\Hc_\Gamma^\varphi,P).
\end{displaymath}
Hence, it remains to choose the operator $P:Z\rightarrow Z^*$ such that it is linear, bounded, symmetric, and coercive. As these properties are equivalent to $\langle P\cdot, \cdot\rangle$ being an inner product on $Z$, we simply search for computationally feasible inner products on $Z$. 
\begin{remark}
\emph{The linear operator~$P$ should obviously only depend on the computations related to one of the space-time subdomains, e.g., $\Omega_2\times \R$, otherwise the associated MDN method does not yield a domain decomposition. Furthermore, the linearity of the operator~$P$ implies that the MDN method is a (linearly convergent) iterative scheme that only requires a linear pre-conditioner for the nonlinear problem $S\eta=0$. This is not the case for the original domain decomposition method~\cref{eq:dn}.}
\end{remark}
A first possible method is given below.
\begin{method}
The solution of $S\eta=0$ in $Z^*$ is approximated by the iteration
\begin{equation}\label{method:1}
    \eta^{n+1}=\eta^n+s P_1^{-1}(\Hc_\Gamma^\varphi)^*(0-S\eta^n), 
\end{equation}
where $(\varphi,s)$ are positive parameters, $\eta^0$ is an initial guess, and the operator $P_1:Z\rightarrow Z^*$ is given by
\begin{displaymath}
\langle P_1\eta, \mu\rangle=\int_{\R}\int_{\Omega_2}\partial_+^{1/2}R_2\eta\partial_+^{1/2}R_2\mu+\nabla R_2\eta\cdot\nabla R_2\mu\,\,\mathrm{d}x\mathrm{d}t\quad\text{for all }\eta,\mu\in Z.
\end{displaymath}
\end{method}
\begin{remark}\label{rem:B}
\emph{We are free to choose any linear right inverse $R_2$ of the trace operator $T_2$ in the method above. However, the specific choice $R_2=B$, where 
$B:\eta\mapsto w_2$ is the solution operator for the equation
\begin{displaymath}
\int_{\R}\int_{\Omega_2}\partial_+^{1/2}w_2\partial_+^{1/2}v+\nabla w_2\cdot\nabla v\,\mathrm{d}x\mathrm{d}t=0\quad\textrm{for all } v\in  W_2^0
\end{displaymath}
with $T_2w_2=\eta$, yields that $P_1$ becomes invariant to the choice of $R_2$ in the second argument. That is,
\begin{displaymath}
\langle P_1\eta, \mu\rangle=\int_{\R}\int_{\Omega_2}\partial_+^{1/2}B\eta\partial_+^{1/2}\tilde{R}_2\mu+\nabla B\eta\cdot\nabla \tilde{R}_2\mu\,\,\mathrm{d}x\mathrm{d}t
\end{displaymath}
for every $\tilde{R}_2$; compare with~\cref{rem:sp}. Note that this possibility to extend $\eta$ and $\mu$ to $W_2$ in different ways enables more efficient implementations of the method.}
\end{remark}
A second approach is to treat the spatial and temoral terms differently when constructing the operator~$P$. One does not even need to employ parabolic extensions of the functions~$\eta$ and~$\mu$. To illustrate this, we introduce a temporal quarter-derivative on $\Gamma\times\R$ as 
\begin{displaymath} 
\partial^{1/4}=\hat{\partial}^{1/4}\tensor I:H^{1/4}(\R)\tensor L^2(\Gamma)\rightarrow L^2(\Gamma\times\R),
\end{displaymath} 
where $\hat{\partial}^{1/4}=\hat{\F}^{-1}\hat{M}_{1/4}\hat{\F}$ and $\hat{M}_{1/4}v(\xi)=(\mathrm{i}\xi)^{1/4} v(\xi)$. Note that this choice is not unique; compare with $\hat{\partial}^{1/2}_+$ and $\hat{\partial}^{1/2}_-$ defined in~\cref{eq:halftime}. Next, we introduce an elliptic extension to $\Omega_2\times\R$ via
\begin{displaymath} 
H_2= I \tensor \hat{H}_2:L^2(\R)\tensor \Lambda\rightarrow L^2(\R)\tensor V_2,
\end{displaymath} 
where $\hat{H}_2:\Lambda\to V_2:\eta\mapsto w_2$ is the solution operator for the (weak) linear elliptic equation 
\begin{displaymath}
   \int_{\Omega_2}\nabla w_2\cdot\nabla v\,\mathrm{d}x=0\quad\textrm{for all } v\in V_2
\end{displaymath}
with $\hat{T}_2w_2=\eta$. As for previous extended operators, \cref{lemma:tensor} yields that $\partial^{1/4}$ and $H_2$ are both linear bounded operators. Furthermore, $H_2$ is a right inverse to the trace operator with the ``larger  domain'' $T_2=I \tensor \hat{T}_2:L^2(\R)\tensor V_2\to L^2(\R)\tensor\Lambda$. We can now define the following method.
\begin{method}
The solution of $S\eta=0$ in $Z^*$ is approximated by the iteration
\begin{equation}\label{method:2}
    \eta^{n+1}=\eta^n+s P_2^{-1}(\Hc_\Gamma^\varphi)^*(0-S\eta^n),
\end{equation}
where $(\varphi,s)$ are positive parameters, $\eta^0$ is an initial guess, and the operator $P_2:Z\rightarrow Z^*$ is given by
\begin{displaymath}
\langle P_2\eta, \mu\rangle=\int_{\R}\int_{\Gamma}\partial^{1/4}\eta\partial^{1/4}\mu\,\mathrm{d}x\mathrm{d}t
                                            +\int_{\R}\int_{\Omega_2}\nabla H_2\eta\cdot\nabla H_2\mu\,\mathrm{d}x\mathrm{d}t\quad\text{for all }\eta,\mu\in Z.
\end{displaymath}
\end{method}
\begin{remark}
\emph{The same invariance as described in~\cref{rem:B} holds true for the operator $P_2$, i.e., the term~$H_2\mu$ can be replaced by any element $\tilde{R}_2\mu$.}
\end{remark}
\begin{remark}
\emph{Another natural method choice would simply be to set $P_3:\eta\mapsto (\eta,\cdot)_Z$. This is theoretically possible, but we are unaware of any efficient way to implement the $1/2$-derivatives related to $\Lambda$ for a nontrivial spatial interface $\Gamma$.}
\end{remark}
\begin{lemma}\label{lemma:P}
Let \cref{ass:domains} be valid. Then the operators $P_\ell:Z\to Z^*$, $\ell=1,2$, are linear, bounded, symmetric, and coercive. 
\end{lemma}
\begin{proof}
The operators $P_\ell$ are readily well defined on $Z$, linear, bounded, and symmetric. The only property that is nontrivial is the coercivity. To this end, we recall the equivalent norms in~\cref{eq:eqivnorms}, and observe that 
\begin{displaymath}
\eta\mapsto\bigl(\|\partial^{1/4}\eta\|^2_{L^2(\Gamma\times\R)}+\|\eta\|^2_{L^2(\R)\tensor\Lambda}\bigr)^{1/2}
\end{displaymath}
is an equivalent norm on $Z$. This follows by the same argumentation as for the $W_i$-norm in~\cref{eq:eqivnorms} and the fact that $\|\cdot\|_{L^2(\R)\tensor\Lambda}$ contains an $\|\cdot\|_{L^2(\Gamma\times\R)}$-term. Moreover, $T_2$ is bounded both when 
interpreted as a mapping from $W_2$ to $Z$, and from $L^2(\R)\tensor V_2$ to $L^2(\R)\tensor\Lambda$. With this we have the inequality
\begin{displaymath}
\langle P_1\eta,\eta\rangle=\|\partial^{1/2}_+R_2\eta\|^2_{L^2(\Omega_i\times\R)}+\|\nabla R_2\eta\|^2_{L^2(\Omega_i\times\R)^d}
     \geq c\|R_2\eta\|^2_{W_i}\geq c\|T_2R_2\eta\|^2_{Z}=c\|\eta\|^2_Z,
\end{displaymath}
i.e., $P_1$ is coercive. Furthermore, we have the bound 
\begin{displaymath}
\|\nabla H_2\eta\|^2_{L^2(\Gamma\times\R)^d}\geq c\|H_2\eta\|^2_{L^2(\R)\tensor V_2}
                                                                          \geq c\|T_2H_2\eta\|^2_{L^2(\R)\tensor \Lambda}
                                                                          = c\|\eta\|^2_{L^2(\R)\tensor \Lambda},
\end{displaymath}
and therefore 
\begin{displaymath}
\langle P_2\eta,\eta\rangle =\|\partial^{1/4}\eta\|^2_{L^2(\Gamma\times\R)}+\|\nabla H_2\eta\|^2_{L^2(\Gamma\times\R)^d}\\
                                           \geq c\bigl(\|\partial^{1/4}\eta\|^2_{L^2(\Gamma\times\R)}+\|\eta\|^2_{L^2(\R)\tensor \Lambda}\bigr)\\
                                           \geq c\|\eta\|^2_Z.
\end{displaymath}
Thus, $P_2$ is also coercive.
\end{proof}
Applying these results gives us the linear convergence result below.
\begin{theorem}\label{cor:mdnconv}
Let \cref{ass:domains,ass:eq,ass:f} be valid. For sufficiently small positive parameters $(\varphi,s)$ and any initial guess $\eta^0\in Z$, the iterates $\{\eta^n\}_{n\geq1}$ of the modified Dirichlet--Neumann methods~\cref{method:1,method:2} are well defined and converge linearly in $Z$ to $\eta$, the solution of the weak Steklov--Poincaré equation~\cref{eq:speq}. Moreover, the iterates $\{F_i\eta^n\}_{n\geq1}$ converges linearly in $W_i$ to $\restr{u}{ \Omega_i\times\R}$, where $u$ is the solution of the weak equation~\cref{eq:weak}.
\end{theorem}
\begin{proof}
The statement regarding the convergence of $\{\eta^n\}_{n\geq1}$ follows directly by~\cref{thm:zarantello} together with~\cref{lemma:sp,lemma:P}. The convergence of $\{F_i\eta^n\}_{n\geq1}$ follows from~\cref{lemma:Fi_lip}, as
\begin{displaymath}
        \|F_i\eta^n-F_i\eta\|_{W_i}\leq C\|\eta^n-\eta\|_Z\leq CL^n\|\eta^0-\eta\|_Z.
\end{displaymath}
Note that by~\cref{lemma:tpspeq,lemma:tranequiv}, one has the identity $(\restr{u}{ \Omega_1\times\R}, \restr{u}{ \Omega_2\times\R})=(F_1\eta, F_2\eta)$.
\end{proof}

\section{Convergence of the Robin--Robin method}\label{sec:conv}

Another way to construct a domain decomposition method is as follows. Instead of alternating between the Dirichlet and Neumann transmission conditions in~\cref{eq:TPStrong}, as done for the Dirichlet--Neumann and the Neumann--Neumann methods~\cite{quarteroni}, one can reformulate the transmission conditions into the Robin conditions 
\begin{displaymath}
\alpha(\nabla u_1)\cdot\nu_i+su_1=\alpha(\nabla u_2)\cdot\nu_i+su_2\quad\text{on }\Gamma\times\R\text{ for }i=1,2,
\end{displaymath}
where $s$ is a positive parameter. Alternating between the subdomains $\Omega_i\times\R$, $i=1,2$, then leads to the \emph{Robin--Robin} method, which has the interface formulation
\begin{equation}\label{eq:pr}
		 \eta^{n+1}=(sJ+S_2)^{-1}(sJ-S_1)(sJ+S_1)^{-1}(sJ-S_2)\eta^n,
\end{equation}
where $J:\eta\mapsto(\eta, \cdot)_{L^2(\Gamma\times\R)}$ is the Riesz isomorphism on $L^2(\Gamma\times\R)$. The derivation of~\cref{eq:pr} is identical to the one used for the setting of nonlinear elliptic equations given in~\cite[Section~6]{EHEE22}. It should also be noted that this reformulation of the Robin--Robin method, first proposed in~\cite{Agoshkov83}, is still non-standard in the literature. As the operators $(\Hc_\Gamma^\varphi)^*(sJ+S_i)$ are Lipschitz continuous and uniformly monotone, which follows by the same argument as in~\cref{lemma:sp}, they are also bijective by~\cref{thm:zarantello}. Thus, the interface formulation~\cref{eq:pr} is well defined on~$Z$.

For quasilinear parabolic equations with iterates fulfilling $F_i\eta^n\in H^1(\R)\tensor L^2(\Omega_i)$, convergence has been derived in~\cite{gander23}. In general, this additional regularity of the iterates~$\eta^n$ in turn requires higher regularity of the (method defined) subdomains $\Omega_i$, $i=1, 2$. However, the latter is not necessarily obtained even for simple domain decompositions. For example, consider~\cref{fig:spacetimecylinder}, where a trivial decomposition of the smooth convex domain~$\Omega$ generates two non-convex subdomains $\Omega_i$ with corners. 

Hence, we aim to prove convergence without assuming additional regularity on~$(\eta^n,\Omega_i)$. This is straightforward, as we have already derived the fundamental properties in~\cref{sec:SP} for the nonlinear time-dependent operators~$S_i,S$. With these abstract results in place, the convergence analysis for the Robin--Robin method applied to quasilinear parabolic equations follows by the same abstract arguments as for the method applied to nonlinear elliptic equations~\cite[Section~8]{EHEE22}. We therefore proceed with a short summary of the main ideas of the abstract convergence proof. 

The formulation~\cref{eq:pr} of the Robin--Robin method with operators mapping $Z$ into $Z^*$ is slightly to general for a convergence proof, and we instead interpret the Steklov--Poincaré operators as unbounded operators on $L^2(\Gamma\times\R)$. To this end, consider the Gelfand triple
\begin{displaymath}
Z\hookrightarrow L^2(\Gamma\times\R)\cong L^2(\Gamma\times\R)^*\hookrightarrow Z^*,
\end{displaymath}
which is well defined by~\cref{lemma:density}. Next, define the restricted operator domains 
\begin{displaymath}
 D(\mathcal{S}_i)=\{\eta\in Z: S_i\eta\in L^2(\Gamma\times\R)^*\},\quad D(\mathcal{S})=\{\eta\in Z: S\eta\in L^2(\Gamma\times\R)^*\},
\end{displaymath}
together with the unbounded nonlinear operators
\begin{align*}
	\mathcal{S}_i&:D(\mathcal{S}_i)\subseteq L^2(\Gamma\times\R)\rightarrow L^2(\Gamma\times\R): \eta\mapsto J^{-1}S_i\eta,\\
	\mathcal{S}&:D(\mathcal{S})\subseteq L^2(\Gamma\times\R)\rightarrow L^2(\Gamma\times\R): \eta\mapsto J^{-1}S\eta.
\end{align*}
The $L^2$-Steklov--Poincaré equation then becomes to find $\eta\in D(\mathcal{S})$ such that 
\begin{equation}\label{eq:spl2}
	  \mathcal{S}\eta=0\quad\text{in }L^2(\Gamma\times\R),
\end{equation}
and the numerical method takes the following form.\\

\noindent{\bf Robin--Robin method}\quad\emph{The solution of~\cref{eq:spl2} is approximated by the iteration
\begin{equation}\label{eq:prl2}
\eta^{n+1}=(sI+\mathcal{S}_2)^{-1}(sI-\mathcal{S}_1)(sI+\mathcal{S}_1)^{-1}(sI-\mathcal{S}_2)\eta^n,
\end{equation}
where $\eta^0\in D(\mathcal{S}_2)$ is a given initial guess and $s>0$ is a method parameter.}\\

As already observed, the operators $S$ and $sJ+S_i$ are all bijective, which implies that the same holds for the operators $\mathcal{S}$ and $sI+\mathcal{S}$. Hence, there exist a unique solution to~\cref{eq:spl2} and the iteration~\cref{eq:prl2} is well defined. For the convergence analysis, we also require the following mild regularity of the solution to the weak parabolic equation~\cref{eq:weak}. 
\begin{assumption}\label{ass:regularity}
The functionals  
\begin{displaymath}
\mu\mapsto a_i(\restr{u}{\Omega_i\times\R}, R_i\mu)-(f_i, R_i\mu)_{L^2(\Omega_i\times\R)},\quad{i=1,2},
\end{displaymath}
are elements in $L^2(\Gamma\times\R)^*$, where $u\in W$ is the solution to~\cref{eq:weak}.
\end{assumption}
\begin{remark}
\emph{The assumption is somewhat implicit, but can be interpreted as the solution having a generalized normal derivative
\begin{displaymath}
    \alpha(u, \nabla u)\cdot \nu_i
\end{displaymath}
on the space-time interface belonging to $L^2(\Gamma\times\R)$. In the case of the linear heat equation, i.e., $(\alpha,\beta)=(\nabla u,0)$, this holds if the solution satisfies the additional regularity 
\begin{displaymath}
u\in W\cap L^2\bigl(\R, H^{3/2+\varepsilon}(\Omega)\bigr),\,\varepsilon>0,\quad\text{or}\quad u\in W\cap L^2\bigl(\R,  C^1(\bar{\Omega})\bigr). 
\end{displaymath}
To see this, first observe that a Lipschitz manifold $\partial\Omega_i$ has a normal vector $\nu_i$ in $L^\infty(\partial\Omega_i)^d$. The additional regularity of~$u$ then yields that each term $\partial_j u\,(\nu_i)_j$ of the normal derivative becomes an element in $L^2(\Gamma\times\R)$}.
\end{remark}
Under this assumption one inherits the following additional regularity for the solution to the $L^2$-Steklov--Poincaré equation. 
\begin{lemma}\label{lemma:friedrich}
Let \cref{ass:domains,ass:eq,ass:f,ass:regularity} be valid. If $\eta$ is the solution to~\cref{eq:spl2} then $\eta\in D(\mathcal{S}_1)\cap D(\mathcal{S}_2)$.
\end{lemma}

We can now prove that the Robin--Robin method converges. The following theorem employs the uniform monotonicity of $S_i$ in~$L^2(\R)\tensor \Lambda$ and the fact that the solution to~\cref{eq:spl2} satisfies $\eta\in D(\mathcal{S}_1)\cap D(\mathcal{S}_2)$, which holds by~\cref{lemma:sp,lemma:friedrich}. The convergence then follows by the abstract result~\cite[Proposition 1]{lionsmercier}. A simpler proof of this abstract result can be found in~\cite[Lemma~8.8]{EHEE22}.
\begin{theorem}\label{thm:rrconv}
Let \cref{ass:domains,ass:eq,ass:f,ass:regularity} be valid. For any parameters~$s>0$ and initial guess~$\eta^0\in D(\mathcal{S}_2)$, the iterates $\{\eta^n\}_{n\geq1}$ of the Robin--Robin method~\cref{eq:prl2} are well defined and converges in $L^2(\R)\tensor\Lambda$ to $\eta$, the solution of the $L^2$-Steklov--Poincaré equation~\cref{eq:spl2}. Furthermore, the iterates $\{F_i\eta^n\}_{n\geq1}$ converge in $L^2(\R)\tensor H^1(\Omega_i)$ to $\restr{u}{ \Omega_i\times\R}$, where $u$ is the solution of the weak equation~\cref{eq:weak}.
\end{theorem}
\begin{remark}
\emph{In contrast to the MDN methods~\cref{method:1,method:2}, the Robin--Robin method converges for \emph{any} choice of the method parameter~$s>0$. However, the Robin--Robin method is unlikely to be linearly convergent in the present continuous framework. This is indicated already in the elliptic case, by observing that the discrete method is linearly convergent with an error reduction constant of the form $L=1-\mathcal{O}(\sqrt{h})$; see~\cite{gander06}. That is, a constant~$L$ that deteriorates to one as the spatial discretization parameter $h$ tends to zero.}
\end{remark}

\section{Extension to a space-time finite element method}\label{sec:FEM}

We will now illustrate how our analysis can be applied when combining domain decompositions and space-time finite element discretizations. As a very first proof of concept, we employ the spectral tensor basis of~\cite{Dahlgren04}, which we detail here. See also~\cite{zank21} for a similar discretization based on a modified Hilbert transform on finite time intervals. 

For the spatial discretization we use piecewise linear basis functions $\{\phi_k\}^M_{k=1}$ given on a suitable triangular partition~$K_h$ of the spatial domain $\Omega$. Here,  $h>0$ denotes the largest diameter in the partition. We denote the spaces spanned by the spatial basis on $\Omega, \Omega_i$ by $V^h, V_i^h$, respectively. For the temporal discretization we use the following spectral basis. Let $N\in\N$, $\tau>0$, and consider the spectral domain $(-N\tau, N\tau)$ with the spectral grid points $\omega_j=j\tau$, $j=-N, \dots, N$. We define the first $N+1$ basis elements $\{\psi_j\}^N_{j=0}$, through their Fourier transforms~$\hat{\F}\psi_j$. They are the unique piecewise linear functions on $(-N\tau, N\tau)$ with respect to the grid above defined by
\begin{displaymath}
    \hat{\F}\psi_j(\omega_\ell)=
        \begin{cases}
             1 &\text{if } j=\ell \text{ or } j=-\ell, \\
             0 &\text{otherwise.}
        \end{cases}
\end{displaymath}
The second set of $N+1$ basis elements $\{\tilde{\psi}_j\}^N_{j=0}$ are defined as
\begin{displaymath}
    \tilde{\psi}_j=\hat{\Hc}\psi_j.
\end{displaymath}
The Fourier transforms $\hat{\F}\tilde{\psi}_j$ can be computed explicitly using~\cref{eq:hilbertfourier}, which yields that
\begin{displaymath}
    \hat{\F}\tilde{\psi}_j(\omega)=-\mathrm{i}\,\sgn(\omega)\hat{\F}\psi_j(\omega).
\end{displaymath}

Note that the first set of basis element is composed of even real functions and the second set of odd imaginary functions, which implies that the inverse Fourier transform is real-valued. In fact, the basis functions are
\begin{align*}
    \psi_0(t)&=\frac{1}{\pi t^2\tau}\bigl(1-\cos(t\tau)\bigr),\\
    \psi_j(t)&=\frac{2}{\pi t^2\tau}\bigl(1-\cos(t\tau)\bigr)\cos(tj\tau),\quad j=1,\dots, N,\\
    \tilde{\psi}_0(t)&=\frac{t\tau-\sin(t\tau)}{\pi t^2\tau},\\
    \tilde{\psi}_j(t)&=\frac{2}{\pi t^2\tau}\bigl(1-\cos(t\tau)\bigr)\sin(tj\tau),\quad j=1,\dots, N.
\end{align*}
We denote the space spanned by these basis elements by $U_N^\tau\subset H^{1/2}(\R)$. Since $\hat{\Hc}\hat{\Hc}=-I$, the space $U_N^\tau$ is invariant under the Hilbert transform. Moreover, since the basis functions are localized in Fourier space, the discretization leads to sparse matrices that can be easily assembled using Parseval's formula. In particular, the bilinear forms containing fractional derivatives are explicitly given as
\begin{align*}
    (\partial^{1/2}_+u,\partial^{1/2}_-v)&=\int_\R i\xi\hat{\F}u\,\overline{\hat{\F}v}\,\mathrm{d}\xi,\\
    (\partial^{1/2}_+u,\partial^{1/2}_+v)&=\int_\R |\xi|\hat{\F}u\,\overline{\hat{\F}v}\,\mathrm{d}\xi,\\
    (\partial^{1/4}u,\partial^{1/4}v)&=\int_\R \sqrt{|\xi|}\hat{\F}u\,\overline{\hat{\F}v}\,\mathrm{d}\xi
\end{align*}
for all $u,v\in U_N^\tau$.

We can then define the full tensor spaces~$W^h=U_N^\tau\otimes V^h\subset W$ and~$W_i^h=U_N^\tau\otimes V_i^h\subset W_i$. Note that for notational purposes we leave out the dependence on $(\tau, N)$ in all of our discrete spaces, operators, and functions, e.g., we write~$W^h$ instead of~$W_N^{h, \tau}$. For the sake of simplicity, we introduce the discrete trace space $Z^h$ via the assumption below.
\begin{assumption}\label{ass:disc}
The spatial partition $K_h$ of $\Omega$ is chosen such that $Z_h=T_1W_1^h=T_2W_2^h$.
\end{assumption}

The discrete weak equation and the discrete transmission problem can then be introduced by simply replacing the function spaces in~\cref{eq:weak} and~\cref{eq:weaktran}, respectively, by their discrete counterparts. The discrete time-dependent Steklov--Poincaré operators $S^h_i:Z^h\rightarrow (Z^h)^*$ are defined as
\begin{displaymath}
\langle S^h_i\eta, \mu\rangle=a_i(F_i^h\eta, R_i^h\mu)-(f_i, R_i^h\mu)_{L^2(\Omega_i\times\R)}.
\end{displaymath}
Here, $F_i^h:Z^h\rightarrow W_i^h$ is the discrete solution operator and $R_i^h:Z^h\rightarrow W_i^h$ is an arbitrary extension operator. With the operators~$S_i^h$ in place, it is straightforward to define the discrete variants of the domain decomposition methods~\cref{method:1,method:2}. The following convergence results hold for these discrete methods. 
\begin{theorem}\label{thm:mdnconvh}
Let \cref{ass:domains,ass:eq,ass:f,ass:disc} be valid. For sufficiently small positive parameters $(\varphi,s)$ and an initial guess $\eta_h^0\in Z^h$, the iterates $\{\eta_h^n\}_{n\geq1}$ of the discrete versions of the MDN methods~\cref{method:1,method:2} are well defined and converge linearly in $Z^h$ to $\eta_h$, the solution of the discrete Steklov--Poincaré equation. Moreover, the iterates $\{F^h_i\eta_h^n\}_{n\geq1}$ converges linearly in $W^h_i$ to $\restr{u_h}{ \Omega_i\times\R}$, where $u_h$ is the solution of the discrete weak equation. 
\end{theorem}
The proof is the same as for the continuous case, utilizing that the Hilbert transform~$\hat{\Hc}:U_N^\tau\rightarrow U_N^\tau$ is an isomorphism. Furthermore, the same type of discrete extensions can be done for~\cref{thm:rrconv} and thereby establishing the convergence of the discrete version of the Robin--Robin method~\cref{eq:prl2}. 

We conclude with a numerical experiment to illustrate the derived convergence results. For this purpose,
we use the spatial domain $\Omega=(0, 1)$, the decomposition $\Omega_1=(0, 1/2)$, $\Omega_2=(1/2, 1)$, $\Gamma=\{1/2\}$, and the linear heat equation, i.e., $(\alpha,\beta)=(\nabla u,0)$, with the source term
\begin{displaymath}
    f(t, x)=\begin{cases}
           \bigl(e^{-t}-\frac12 e^{-t/2}\bigr)(x^2-x^3)+(e^{-t}-e^{-t/2})(2-6x) &\text{for } t>0, \\
           0 & \text{for } t\leq 0.
    \end{cases}
\end{displaymath}
The exact solution is then 
\begin{displaymath}
    u(t, x)=\begin{cases}
    (e^{-t/2}-e^{-t})(x^2-x^3) &\text{for } t>0, \\
    0 &\text{for }  t\leq 0.
   \end{cases}
\end{displaymath}
We also fix an equidistant spatial partition $K_h$ with $h=1/512$, and a spectral grid with $(N,\tau)=(256,0.5)$. It is then straightforward to validate that~\cref{ass:domains,ass:eq,ass:f,ass:regularity,ass:disc} hold. 

Next, consider the discrete versions of the modified Dirichlet--Neumann methods given in~\cref{rem:B} and~\cref{method:2}, as well as the Robin--Robin method~\cref{eq:prl2}. These methods are hereafter referred to as MDN1, MDN2, and RR. As all three methods are invariant of the choice of the operator $R_i^h$ appearing in the terms $R_i^h\mu$, all such terms are implemented by taking $R_i^h$ to be the extension by zero on the interior degrees of freedom. To obtain an easily computable numerical error, observe that all methods converge in $L^2(\R)\tensor H^1(\Omega_i)$ and therefore also in, e.g., $L^2(0,1)\tensor H^1(\Omega_i)$. Hence, we compute a relative error by comparing with the exact solution $u$ on the finite time interval $(0,1)$, i.e., 
\begin{align*}
    e_e&=\frac{\|u-u_1^h\|_{L^2(0,1)\tensor H^1(\Omega_1)}+\|u-u_2^h\|_{L^2(0,1)\tensor H^1(\Omega_2)}}{\|u\|_{L^2(0,1)\tensor H^1(\Omega_1)}+\|u\|_{L^2(0,1)\tensor H^1(\Omega_2)}}.
\end{align*}
The parameters used are $(\varphi,s)=(0.02\pi,0.55)$, $(\varphi,s)=(0.02\pi,0.7)$, and $s=2.5$ for MDN1, MDN2, and RR, respectively. The results are given in~\cref{fig:exactresults}. 
From these results we see that all three methods initially display an error decay in line with~\cref{thm:mdnconvh}, and after $n=7$ iterations all the errors have reached a constant level. That is, after just a few iterations the domain decomposition errors have decreased to the size of the underlying space-time finite element error.

Further experiments evaluating the efficiency of these methods for nonlinear equations on Lipschitz domains, as well as developing strategies for choosing the method parameters, will be conducted elsewhere. 
\begin{figure}
\centering
\includegraphics[width=0.7\linewidth]{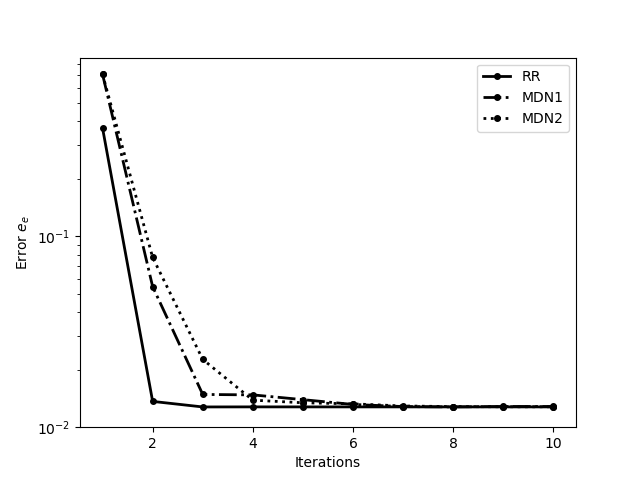}
\caption{The relative errors of the modified Dirichlet--Neumann iterations (MDN1, MDN2) and the Robin--Robin method (RR).}
\label{fig:exactresults}
\end{figure}%

\section*{Acknowledgements}
The authors thank Monika Eisenmann for her helpful input on Zarantello's theorem and for making them aware of the reference~\cite[Theorem 25.B]{zeidler2b}.

\section*{Funding}
This work was supported by the Swedish Research Council under the grant 2019--05396.

\bibliographystyle{plain}
\bibliography{references}

\begin{thebibliography}{10}

\bibitem{Agoshkov83}
Valery~I. Agoshkov and Vyacheslav~I. Lebedev.
\newblock Variational algorithms of the domain decomposition method
  [translation of {P}reprint 54, {A}kad. {N}auk {SSSR}, {O}tdel. {V}ychisl.
  {M}at., {M}oscow, 1983].
\newblock volume~5, pages 27--46. 1990.
\newblock Soviet Journal of Numerical Analysis and Mathematical Modelling.

\bibitem{japhet20}
Elyes Ahmed, Caroline Japhet, and Michel Kern.
\newblock Space-time domain decomposition for two-phase flow between different
  rock types.
\newblock {\em Comput. Methods Appl. Mech. Engrg.}, 371:113294, 30, 2020.

\bibitem{aubin}
Jean-Pierre Aubin.
\newblock {\em Applied functional analysis}.
\newblock Pure and Applied Mathematics (New York). Wiley-Interscience, New
  York, second edition, 2000.

\bibitem{Halpern10}
Filipa Caetano, Martin~J. Gander, Laurence Halpern, and J\'{e}r\'{e}mie
  Szeftel.
\newblock Schwarz waveform relaxation algorithms for semilinear
  reaction-diffusion equations.
\newblock {\em Netw. Heterog. Media}, 5(3):487--505, 2010.

\bibitem{costabel90}
Martin Costabel.
\newblock Boundary integral operators for the heat equation.
\newblock {\em Integral Equations Operator Theory}, 13(4):498--552, 1990.

\bibitem{Dahlgren04}
Martin Dahlgren.
\newblock A finite element method for parabolic equations.
\newblock In {\em Progress in industrial mathematics at {ECMI} 2002}, volume~5
  of {\em Math. Ind.}, pages 253--258. Springer, Berlin, 2004.

\bibitem{EHEE22}
Emil Engstr\"{o}m and Eskil Hansen.
\newblock Convergence analysis of the nonoverlapping {R}obin-{R}obin method for
  nonlinear elliptic equations.
\newblock {\em SIAM J. Numer. Anal.}, 60(2):585--605, 2022.

\bibitem{fontesthesis}
Magnus Fontes.
\newblock {\em Parabolic equations with low regularity}.
\newblock 1996.
\newblock Thesis, Lund University.

\bibitem{fontes09}
Magnus Fontes.
\newblock Initial-boundary value problems for parabolic equations.
\newblock {\em Ann. Acad. Sci. Fenn. Math.}, 34(2):583--605, 2009.

\bibitem{gander06}
Martin~J. Gander.
\newblock Optimized {S}chwarz methods.
\newblock {\em SIAM J. Numer. Anal.}, 44(2):699--731, 2006.

\bibitem{gander15}
Martin~J. Gander.
\newblock 50 years of time parallel time integration.
\newblock In {\em Multiple shooting and time domain decomposition methods},
  volume~9 of {\em Contrib. Math. Comput. Sci.}, pages 69--113. Springer, Cham,
  2015.

\bibitem{gander07}
Martin~J. Gander and Laurence Halpern.
\newblock Optimized {S}chwarz waveform relaxation methods for advection
  reaction diffusion problems.
\newblock {\em SIAM J. Numer. Anal.}, 45(2):666--697, 2007.

\bibitem{kwok16}
Martin~J. Gander, Felix Kwok, and Bankim~C. Mandal.
\newblock Dirichlet-{N}eumann and {N}eumann-{N}eumann waveform relaxation
  algorithms for parabolic problems.
\newblock {\em Electron. Trans. Numer. Anal.}, 45:424--456, 2016.

\bibitem{kwok21}
Martin~J. Gander, Felix Kwok, and Bankim~C. Mandal.
\newblock Dirichlet-{N}eumann waveform relaxation methods for parabolic and
  hyperbolic problems in multiple subdomains.
\newblock {\em BIT}, 61(1):173--207, 2021.

\bibitem{gander23}
Martin~J. Gander, Stephan~B. Lunowa, and Christian Rohde.
\newblock Non-overlapping {S}chwarz waveform-relaxation for nonlinear
  advection-diffusion equations.
\newblock {\em SIAM J. Sci. Comput.}, 45(1):A49--A73, 2023.

\bibitem{keller02}
Eldar Giladi and Herbert~B. Keller.
\newblock Space-time domain decomposition for parabolic problems.
\newblock {\em Numer. Math.}, 93(2):279--313, 2002.

\bibitem{Halpern12}
Laurence Halpern, Caroline Japhet, and J\'{e}r\'{e}mie Szeftel.
\newblock Optimized {S}chwarz waveform relaxation and discontinuous {G}alerkin
  time stepping for heterogeneous problems.
\newblock {\em SIAM J. Numer. Anal.}, 50(5):2588--2611, 2012.

\bibitem{japhet13}
Thi-Thao-Phuong Hoang, J\'{e}r\^{o}me Jaffr\'{e}, Caroline Japhet, Michel Kern,
  and Jean~E. Roberts.
\newblock Space-time domain decomposition methods for diffusion problems in
  mixed formulations.
\newblock {\em SIAM J. Numer. Anal.}, 51(6):3532--3559, 2013.

\bibitem{japhet16}
Thi-Thao-Phuong Hoang, Caroline Japhet, Michel Kern, and Jean~E. Roberts.
\newblock Space-time domain decomposition for reduced fracture models in mixed
  formulation.
\newblock {\em SIAM J. Numer. Anal.}, 54(1):288--316, 2016.

\bibitem{hytonen}
Tuomas Hyt\"{o}nen, Jan van Neerven, Mark Veraar, and Lutz Weis.
\newblock {\em Analysis in {B}anach spaces. {V}ol. {I}. {M}artingales and
  {L}ittlewood-{P}aley theory}, volume~63 of {\em Results in Mathematics and
  Related Areas. 3rd Series. A Series of Modern Surveys in Mathematics]}.
\newblock Springer, Cham, 2016.

\bibitem{king}
Frederick~W. King.
\newblock {\em Hilbert transforms. {V}ol. 1}, volume 124 of {\em Encyclopedia
  of Mathematics and its Applications}.
\newblock Cambridge University Press, Cambridge, 2009.

\bibitem{kufner}
Alois Kufner, Old\v{r}ich John, and Svatopluk Fu\v{c}\'{\i}k.
\newblock {\em Function spaces}.
\newblock Noordhoff International Publishing, Leyden; Academia, Prague, 1977.

\bibitem{Langer21}
Ulrich Langer and Marco Zank.
\newblock Efficient direct space-time finite element solvers for parabolic
  initial-boundary value problems in anisotropic {S}obolev spaces.
\newblock {\em SIAM J. Sci. Comput.}, 43(4):A2714--A2736, 2021.

\bibitem{larsson15}
Stig Larsson and Christoph Schwab.
\newblock Compressive space-time galerkin discretizations of parabolic partial
  differential equations.
\newblock Technical Report 2015-04, Seminar for Applied Mathematics, ETH
  Z{\"u}rich, Switzerland, 2015.

\bibitem{lemarie13a}
Florian Lemari\'{e}, Laurent Debreu, and Eric Blayo.
\newblock Toward an optimized global-in-time {S}chwarz algorithm for diffusion
  equations with discontinuous and spatially variable coefficients. {P}art 1:
  {T}he constant coefficients case.
\newblock {\em Electron. Trans. Numer. Anal.}, 40:148--169, 2013.

\bibitem{lemarie13b}
Florian Lemari\'{e}, Laurent Debreu, and Eric Blayo.
\newblock Toward an optimized global-in-time {S}chwarz algorithm for diffusion
  equations with discontinuous and spatially variable coefficients. {P}art 2:
  {T}he variable coefficients case.
\newblock {\em Electron. Trans. Numer. Anal.}, 40:170--186, 2013.

\bibitem{lionsmagenes1}
Jacques-Louis Lions and Enrico Magenes.
\newblock {\em Non-homogeneous boundary value problems and applications. {V}ol.
  {I}}.
\newblock Die Grundlehren der mathematischen Wissenschaften, Band 181.
  Springer-Verlag, New York-Heidelberg, 1972.

\bibitem{lionsmagenes2}
Jacques-Louis Lions and Enrico Magenes.
\newblock {\em Non-homogeneous boundary value problems and applications. {V}ol.
  {II}}.
\newblock Die Grundlehren der mathematischen Wissenschaften, Band 182.
  Springer-Verlag, New York-Heidelberg, 1972.

\bibitem{lionsmercier}
Pierre-Louis Lions and Bertrand Mercier.
\newblock Splitting algorithms for the sum of two nonlinear operators.
\newblock {\em SIAM J. Numer. Anal.}, 16(6):964--979, 1979.

\bibitem{lucchetti}
Roberto Lucchetti and Fioravante Patrone.
\newblock On {N}emytskii's operator and its application to the lower
  semicontinuity of integral functionals.
\newblock {\em Indiana Univ. Math. J.}, 29(5):703--713, 1980.

\bibitem{quarteroni}
Alfio Quarteroni and Alberto Valli.
\newblock {\em Domain decomposition methods for partial differential
  equations}.
\newblock Numerical Mathematics and Scientific Computation. The Clarendon
  Press, Oxford University Press, New York, 1999.

\bibitem{steinbach19}
Olaf Steinbach and Huidong Yang.
\newblock Space-time finite element methods for parabolic evolution equations:
  discretization, a posteriori error estimation, adaptivity and solution.
\newblock In {\em Space-time methods---applications to partial differential
  equations}, volume~25 of {\em Radon Ser. Comput. Appl. Math.}, pages
  207--248. De Gruyter, Berlin, 2019.

\bibitem{steinbach20}
Olaf Steinbach and Marco Zank.
\newblock Coercive space-time finite element methods for initial boundary value
  problems.
\newblock {\em Electron. Trans. Numer. Anal.}, 52:154--194, 2020.

\bibitem{tartar}
Luc Tartar.
\newblock {\em An introduction to {S}obolev spaces and interpolation spaces},
  volume~3 of {\em Lecture Notes of the Unione Matematica Italiana}.
\newblock Springer, Berlin; UMI, Bologna, 2007.

\bibitem{widlund}
Andrea Toselli and Olof Widlund.
\newblock {\em Domain decomposition methods---algorithms and theory}, volume~34
  of {\em Springer Series in Computational Mathematics}.
\newblock Springer-Verlag, Berlin, 2005.

\bibitem{weidmann}
Joachim Weidmann.
\newblock {\em Linear operators in {H}ilbert spaces}, volume~68 of {\em
  Graduate Texts in Mathematics}.
\newblock Springer-Verlag, New York-Berlin, 1980.

\bibitem{zank21}
Marco Zank.
\newblock An exact realization of a modified {H}ilbert transformation for
  space-time methods for parabolic evolution equations.
\newblock {\em Comput. Methods Appl. Math.}, 21(2):479--496, 2021.

\bibitem{zeidler2b}
Eberhard Zeidler.
\newblock {\em Nonlinear functional analysis and its applications. {II}/{B}}.
\newblock Springer-Verlag, New York, 1990.
\newblock Nonlinear monotone operators, Translated from the German by the
  author and Leo F. Boron.

\end{thebibliography}

\end{document}